\documentclass[reqno,12pt]{amsart}
\usepackage{amsmath,amssymb,epsfig,color,here,wasysym,stmaryrd,slashbox,tikz-cd,hyperref}

\numberwithin{equation}{section}

\newtheorem{theorem}{Theorem}[section]

\newtheorem{lemma}[theorem]{Lemma}
\newtheorem{proposition}[theorem]{Proposition}
\newtheorem{definition}[theorem]{Definition}
\newtheorem{corollary}[theorem]{Corollary}

\theoremstyle{remark}
\newtheorem{example}[theorem]{Example}
\newtheorem{remark}[theorem]{Remark}

\newcommand{\rest}[2] {#1\left| {_{#2}}\right. }
\newcommand{\vanish}[1]{}

\newcommand{\Ss}[1]{\operatorname{\bf S}_{#1}}
\newcommand{\Aa}[1]{\operatorname{\bf A}_{#1}}
\newcommand{\Rr}[1]{\operatorname{\bf R}_{#1}}
\newcommand{\FF}{\operatorname{F}}
\newcommand{\GG}{\operatorname{G}}
\newcommand{\HH}{\operatorname{H}}
\newcommand{\TT}{\operatorname{T}}

\begin{document}

\title{Hypermaps with hyperedges of length at most $3$}

\author[Robert Cori and G\'abor Hetyei]{Robert Cori \and G\'abor Hetyei}

\address{Labri, Universit\'e Bordeaux 1, 33405 Talence Cedex, France.
\hfill\break
WWW: \tt http://www.labri.fr/perso/cori/.}

\address{Department of Mathematics and Statistics,
  UNC Charlotte, Charlotte NC 28223-0001.
WWW: \tt http://webpages.uncc.edu/ghetyei/.}

\date{\today}


\begin{abstract}
We study the computation of our
recently introduced Whitney polynomial and the enumeration of the
spanning hypertrees for hypermaps whose hyperedges have length at most $3$. This
is a class of hypermaps where the computation of the above invariants
depends only on the underlying (multi)hypergraph structure. We develop
deletion-contraction formulas involving six types of generalized loops
and bridges, and we prove results on special substitutions into our Whitney
polynomial. We generalize the reliability polynomial and the random
cluster model to hypermaps in general in such a way that they can be
computed using our Whitney polynomial. Finally we explicitly count the
spanning hypertrees in reciprocals of plane graphs in which every vertex
has degree at most $3$. 
  
\end{abstract}

\maketitle

\section{Introduction}
In a recent paper~\cite{Cori-Hetyei2} the present authors introduced a
Whitney polynomial for hypermaps. The Whitney polynomial $R(G;u,v)$ of a
graph is connected to its Tutte polynomial $T(G;x,y)$ via the formula
\begin{equation}
\label{eq:RT}  
R(G;u,v)=T(G;u+1,v+1).
\end{equation}
A hypermap is a pair of permutations $(\sigma,\alpha)$ whose
cycles represent the vertices and its hyperedges, but each hypermap may
also be thought of as a generalized hypergraph
topologically embedded in an orientable surface. The Whitney polynomial
we introduced satisfies some nice deletion-contraction formulas, which
allow the possibility of partially deleting and contracting a
hyperedge. A key notion for our definition of our Whitney polynomial is
the {\em refinement} where 
we replace each cycle of $\alpha$ (i.e., the hyperedges) with a set of
smaller pairwise noncrossing cycles, each such set of cycles forms a
{\em noncrossing 
  partition} as defined by Kreweras~\cite{Kreweras}. When a hypermap is
a {\em map}, i.e., a graph topologically embedded in an orientable
surface, the Whitney polynomial is an easy substitution away from the
Tutte polynomial, and the Tutte-Whitney invariant depends only on the
underlying graph, not on the way it is drawn on a surface. We have shown
by an example~\cite[Example~2.4]{Cori-Hetyei2} that allowing hyperedges
of length $4$ makes the computation of our Whitney polynomial dependent
on the topological structure: two hypermaps with the same underlying
hypergraph may have different Whitney polynomials. It is not possible to
construct such an example if we limit the size of each hyperedge to at
most~$3$. The subject of this paper is the study of such hypermaps,
which we call hypermaps with {\em short} hyperedges. The class of such
hypermaps is the largest class for which our generalization of the
Whitney polynomial of a graph is a purely combinatorial notion, depends
only on the incidence relations between hyperedges and vertices, and it
is independent of the topological structure, encoded by the cyclic order
of the points in the cycles representing the vertices and the
hyperedges. 

Our paper is structured as follows. In the Preliminaries we review the
relevant definitions and pertinent results on (collections of) hypermaps
and their Whitney polynomials. In Sections~\ref{sec:shorthc} and
\ref{sec:rec} we generalize the famous deletion-contraction formulas
stated in Corollary~\ref{cor:wrg2rec} for 
maps and graphs to hypermaps with short hyperedges. The recurrences
remain the same for ordinary edges of length two, and in
Section~\ref{sec:shorthc} we partition the hyperedges
of length three into $5$ classes: ordinary regular, simple loop, simple
bridge, double loop and double bridge. The recurrence of the Whitney
polynomials for these $5$ types of hyperedges is stated in $5$ separate
propositions in Section~\ref{sec:rec}. 

It is well-known that some particular substitutions into the Tutte
polynomial of a graph yield special graph invariants, such as the number
of spanning forests, the number of spanning subsets, the number of
independent sets. Generalizations of some of such results were stated
in~\cite{Cori-Hetyei2}. In Section~\ref{sec:subs} we add a few new
results about hypermaps with short cycles that have no analogue in graph
theory. In particular, in the case when our hypermap contains no
ordinary edge, the Whitney analogue of the trivial result stating that the Tutte
polynomial evaluated at $x=y=0$ is $0$ does not hold anymore. We get a
more complex result instead, stating that the Whitney polynomial of
hypermap having only hyperedges of length $3$
evaluated at $u=v=-1$ is a power of $(-1)$. The exponent is the
difference of the number of vertices and faces divided by $2$.

In Section~\ref{sec:rel} we generalize the reliability polynomial of a
graph to the reliability polynomial of a hypermap in a way that it can
be expressed in terms of the Whitney polynomial. The main result of this
section holds for all hypermaps, but it is likely most interesting for
hypermaps with short hyperedges where the invariants introduced are
combinatorial and not topological.

Finally, Sections~\ref{sec:a2cm}, \ref{sec:sht} and \ref{sec:3rec} are
about counting the spanning hypertrees in a hypermap having short
hyperedges. As stated in~\cite[Proposition~2.15]{Cori-Hetyei2} the
number of spanning hypertrees may be obtained by substituting $u=v=0$
into the Whitney polynomial of the hypermap. This
special instance of our Whitney polynomial is explored in greater depth
in~\cite{Cori-Hetyei}. In Section~\ref{sec:a2cm} we introduce an
operation associating a $2$-colored map to each hypermap with short
edges, and prove some of its properties. We express the number of
spanning hypertrees in terms of this associated $2$-colored map in
Section~\ref{sec:sht}. This formula takes a particularly nice form in
the case of the {\em reciprocals} of plane graphs in which the maximum
degree of a vertex is $3$. The reciprocal of a hypermap is obtained by 
by swapping the roles of the two permutations encoding the vertices and
the hyperedges, respectively. Our final Section~\ref{sec:3rec} contains
some sample calculations using our results.

\section{Preliminaries}

A {\em collection of hypermaps} is a pair $(\sigma,\alpha)$ of permutations
acting on the same set $\{1,2,\ldots,n\}$ of points. The orbits of the
permutation group $\langle \sigma,\alpha\rangle$ generated by $\sigma$
and $\alpha$ are the {\em connected components}, we denote their numbers
by $\kappa(\sigma,\alpha)$. A collection of hypermaps $(\sigma,\alpha)$
is a {\em hypermap} is $\kappa(\sigma,\alpha)=1$, i.e., the permutation
group $\langle \sigma,\alpha\rangle$ is transitive. A hypermap is a {\em
  map} if all hyperedges have at most two points. The notion of a
topological graph (that is, a graph with a fixed cyclic order of the
vertices around each edge) is very close to that of a map: the only
difference is that a map may have {\em buds} that is hyperedges
containing a single point. The cycles of
$\sigma$ are the {\em vertices}, the cycles of $\alpha$ are the {\em
  hyperedges}. The cycles of $\alpha^{-1}\sigma$ are the {\em
  faces}. The following formula, due to Jacques~\cite{Jacques}, allows
to compute the smallest genus $g(\sigma,\alpha)$ of an oriented surface
on which a hypermap $(\sigma,\alpha)$ may be drawn: 
\begin{equation}
\label{eq:genusdef}
n + 2 -2g(\sigma,\alpha) = z(\sigma) + z(\alpha) + z(\alpha^{-1}
\sigma),
\end{equation}
where $z(\pi)$ denotes the number of cycles of the permutation $\pi$. We
extend the definition of the genus to collections of hypermaps by taking
the sum of the genuses of the connected components.
A permutation
$\theta$ is a {\em refinement} of a permutation $\gamma$, if there exists a
pair of decompositions $\gamma=\gamma_1\cdots \gamma_t$  and
$\theta=\theta_1\cdots \theta_t$ such that the $\gamma_i$s are the
cycles of $\gamma$, the $\theta_i$s are products of disjoint cycles of
$\theta$, and for each $i$ the permutations $\gamma_i$ and $\theta_i$ act
on the same set of elements, and they satisfy
$g(\theta_i,\gamma_i)=0$. By a result of Cori~\cite[Theorem~1]{Cori}
this definition is equivalent to requiring that each $\theta_i$ is a
{\em noncrossing partition} with respect to the cyclic order of $\gamma_i$.
A hypermap $(\sigma,\alpha')$ {\em spans}
the hypermap $(\sigma,\alpha)$ if $\alpha'$ is a refinement of
$\alpha$. Note that not all refinements $\alpha'$ of $\alpha$ have the
property that $(\sigma,\alpha')$ is a hypermap. For refinement $\theta$ of
$\alpha$ the hypermap $(\sigma,\theta)$ is a {\em spanning hypertree} if
it is spanning, it is unicellular (that is, $z(\theta^{-1}\sigma)=0$)
and it has genus zero. 

The notions of {\em hyperdeletions} and {\em hypercontractions} were
introduced in~\cite{Cori-Hetyei}. A {\em hyperdeletion} is an operation
replacing the collection of hypermaps $(\sigma,\alpha)$ with
$(\sigma,\alpha (i,j))$ where $i$ and $j$ belong to the same cycle of
$\alpha$, that is, the transposition $(i,j)$ {\em disconnects}
$\alpha$. For maps a hyperdeletion represents the deletion of an edge
in the underlying graph. In~\cite{Cori-Hetyei} only such hyperdeletions on
maps were considered for which $(\sigma,\alpha (i,j))$ is still a
hypermap. The notion may be seamlessly 
extended to collections of hypermaps: even in collections of maps we may
delete an edge that represents a bridge in the underlying graph, thus
increasing the number of connected components by $1$. If $(i,j)$
disconnects $\alpha$, the permutation $\alpha(i,j)$ is a refinement of
$(\alpha)$.

A {\em hypercontraction} is an operation replacing the collection of
hypermaps $(\sigma,\alpha)$ with $((i,j)\sigma,(i,j)\alpha)$ where
$(i,j)$ {\em disconnects} $\alpha$ and $i$ and $j$ belong to different
cycles of $\sigma$, that is $(i,j)$ {\em connects}
$\sigma$. In~\cite{Cori-Hetyei} a more general class of 
hypercontractions was considered, the requirement that $(i,j)$ must
connect $\sigma$ was only stated for {\em topological}
hypercontractions. In this work all hypercontractions will be
topological. For maps a hypercontraction represents the contraction of a
non-loop edge in the underlying graph. This definition may also be
extended to collections of hypermaps: a hypercontraction does not change
the number of connected components or the number of faces. 

The {\em dual} of the collection of hypermaps $(\sigma,\alpha)$ is the
collection of hypermaps\\ $(\alpha^{-1}\sigma,\alpha^{-1})$,
see~\cite{Cori-Penaud}. This notion
of duality generalizes the usual duality of planar graphs, exchanging
vertices and faces. The 
{\em reciprocal} of the collections hypermaps $(\sigma,\alpha)$ is the hypermap
$(\alpha,\sigma)$. Taking the reciprocal generalizes taking the line
graph of a graph. Both operations preserve the number of connected
components, they takes hypermaps into hypermaps, and also preserve their genus.

The {\em Whitney polynomial  $R(\sigma,\alpha;u,v)$} of a
collection of hypermaps $(\sigma,\alpha)$ on a set of $n$ points is
defined in~\cite{Cori-Hetyei2} by the formula 
$$
R(\sigma,\alpha;u,v)=\sum_{\beta\leq\alpha}
u^{\kappa(\sigma,\beta)-\kappa(\sigma,\alpha)}\cdot
v^{\kappa(\sigma,\beta)+n-z(\beta)-z(\sigma)}
$$
Here the summation is over all permutations $\beta$ refining
$\alpha$. This definition generalizes the notion of the Whitney
polynomial of a map. Formulas similar to the deletion-contraction formulas for
the Tutte polynomial of a map also exist. 
Theorems~\ref{thm:wrgrec} and \ref{thm:wrgrecc} below were first stated
and shown in~\cite[Theorems~2.8 and 2.11]{Cori-Hetyei2}.    
\begin{definition}
\label{def:phik}
Let $(\sigma,\alpha)$ be a collection of hypermaps and assume that
$(1,2,\ldots,m)$ is a cycle of $\alpha$ of length at least $2$. For each
$k\in\{1,2,\ldots,m\}$ we define the collection of hypermaps
$\phi_k(H)=(\sigma_k,\alpha_k)$ where
$$
\sigma_k=
\begin{cases}
  \sigma & \mbox{if $1$ and $k$ belong to the same cycle of $\sigma$,}\\
  (1,k)\sigma & \mbox{otherwise;}\\
\end{cases}  
$$
and the permutation $\alpha_k$ is obtained from $\alpha$ by replacing
  the cycle $(1,2,\ldots,m)$  with 
  $(1)(2,\ldots,m)$ if $k\in\{1,2\}$ and 
  $(1)(2,\ldots,k-1)(k,\ldots,m)$ if $k\not\in\{1,2\}$.
\end{definition}

\begin{theorem}[\cite{Cori-Hetyei2}]
\label{thm:wrgrec}  
Let $H=(\sigma,\alpha)$ be a collection of hypermaps on the set
$\{1,2,\ldots,n\}$ and assume that $(1,2,\ldots,m)$ is a cycle of
$\alpha$ of length at least $2$. Then the Whitney polynomial $R(H;u,v)$
is given by the sum
$$
R(H;u,v)=\sum_{k=1}^m R(\phi_k(H);u,v)\cdot w_k,
$$
where each $w_k$ is a monomial from the set $\{1,u,v,uv\}$, given by the
equations 
\begin{equation}
\label{eq:wrule}
w_k=
\begin{cases}
u^{\kappa(\phi_k(H))-\kappa(H)}v & \mbox{if $k\neq 1$ and $1$ and $k$ belong to
  the same cycle of $\sigma$};\\  
u^{\kappa(\phi_k(H))-\kappa(H)}& \mbox{otherwise}.\\
\end{cases}  
\end{equation}
\end{theorem}  

The number of connected components of $\phi_k(H)$ in
Theorem~\ref{thm:wrgrec} above may exceed $\kappa(H)$. The formula in
the next result expresses $R(H;u,v)$ in terms of Whitney polynomials of
collections of hypermaps $\psi_k(H)$ satisfying
$\kappa(H)=\kappa(\psi_k(H))$. In particular the Whitney polynomial of a
hypermap is expressed in terms of Whitney polynomials of hypermaps.
\begin{theorem}[\cite{Cori-Hetyei2}]
\label{thm:wrgrecc}  
Let $H=(\sigma,\alpha)$ be a collection of hypermaps on the set
$\{1,2,\ldots,n\}$ and assume that $(1,2,\ldots,m)$ is a cycle of
$\alpha$. Then the Whitney polynomial $R(H;u,v)$ is given by the sum
$$
R(H;u,v)=\sum_{k=1}^m R(\psi_k(H);u,v)\cdot w_k,
$$
where each monomial $w_k$ is given by the rule~(\ref{eq:wrule}) and each
$\psi_k(H)$ is a collection of hypermaps defined as follows:
$$
\psi_k(H) =
\begin{cases}
((1,k)\sigma,(1,k)\alpha(1,k-1))& \mbox{if $z((1,k)\sigma)\leq z(\sigma)$
    and $\kappa(\phi_k(H))=\kappa(H)$};\\
((1,2)(1,k)\sigma,(1,2)(1,k)\alpha)& \mbox{if $z((1,k)\sigma)\leq z(\sigma)$
    and $\kappa(\phi_k(H))=\kappa(H)+1$};\\
(\sigma,\alpha(1,k-1)(1,m)) & \mbox{if $z((1,k)\sigma)=z(\sigma)+1$
    and $\kappa(\phi_k(H))=\kappa(H)$;}\\ 
((1,2)\sigma,(1,2)\alpha(k-1,m)) &\mbox{if $z((1,k)\sigma)=z(\sigma)+1$
    and $\kappa(\phi_k(H))=\kappa(H)+1$.}\\
\end{cases}  
$$
Here we count modulo $m$.
\end{theorem}
As noted in~\cite[Example~2.9]{Cori-Hetyei2} for maps, in the special case when
$(\sigma,\alpha)$ contains a $2$-cycle, Theorem~\ref{thm:wrgrecc} yields
the familiar deletion-contraction formula.
\begin{corollary}
\label{cor:wrg2rec}  
Let $(\sigma,\alpha)$ be a collection of hypermaps and let $(1,2)$ be
a cycle of $\alpha$. Then we have
$$
R(\sigma,\alpha;u,v)=
\begin{cases}
  (1+u)R((1,2)\sigma,(1,2)\alpha);u,v)&\text{if $e$ is a bridge;}\\
  (1+v)R(\sigma,\alpha(1,2);u,v)&\text{if $e$ is a loop;}\\
  R((1,2)\sigma,(1,2)\alpha)+R(\sigma,\alpha(1,2);u,v)&\text{otherwise}.\\
\end{cases}
$$
\end{corollary}  
As a consequence, the Whitney-polynomial of a map depends only on the
underlying graph structure. This statement can not be generalized for
hypermaps and hypergraphs, for two reasons. First, the hyperedges of a hypermap
may be incident to the same vertex several times, this generalization of
the notion of a ``loop'' is usually not considered in the literature on
hypergraphs. This obstacle can be easily overcome by defining {\em
  multi-hypergraphs} whose hyperedges are multisets of vertices. The
second obstacle is more serious: as shown
in~\cite[Example~2.4]{Cori-Hetyei2} there are two hypermaps containing a
hyperedge of length $4$ having the same 
underlying hypergraph structure but different Whitney polynomials. It is
not possible to construct such an example if we allow only
{\em short} hyperedges, containing at most three points: as shown
in Section~\ref{sec:shorthc}, we can
generalize Corollary~\ref{cor:wrg2rec} to such hypermaps, obtaining
recurrences which remain the same if we replace a hyperedge
$(i_1,i_2,i_3)$ with the hyperedge $(i_1,i_3,i_2)$, leaving the vertex
permutation unchanged. As a consequence, the Whitney polynomial depends
only on the {\em underlying multi-hypergraph structure}, that is, the
multisets of vertices incident to each hyperedge. 

\section{A classification of hyperedges of length $3$}
\label{sec:shorthc}

In preparation of Section~\ref{sec:rec} where we will generalize
Corollary~\ref{cor:wrg2rec} to hypergraphs with short hyperedges, in
this section we generalize the notions of a loop and of a bridge from
edges to hyperedges of length $3$. 

\begin{definition}
We say that a collection of hypermaps $(\sigma,\alpha)$ has {\em short
  hyperedges} if each cycle of $\alpha$ has at most three elements.
\end{definition}

The recursive computation of the Tutte polynomial of a map depends on
the distinction between three types of edges: loops, bridges and
ordinary edges, which are neither loops nor bridges. In this section we
create an analogous classification of the hyperedges of length $3$. We
will use this classification in Section~\ref{sec:rec} to compute the
contribution of each hyperedge to the Whitney polynomial. We begin with
generalizing the notion of a loop. 
\begin{definition}
\label{def:loops}
Let $(\sigma,\alpha)$ be a collection of hypermaps and let $(i,j,k)$ be
a $3$-cycle of $\alpha$. We say that $(i,j,k)$ is a {\em double loop}
if all three of its points belong to the same vertex. We say that
$(i,j,k)$ is a {\em single loop} if exactly two elements of the set $\{i,j,k\}$
belong to the same vertex. We say that $(i,j,k)$ is a {\em regular
  hyperedge} if its points belong to three distinct vertices. 
\end{definition}
Clearly each hyperedge is either regular, or a single loop, or a double
loop, and these possibilities are mutually exclusive. Similarly, we
generalize the notion of a bridge as follows.
\begin{definition}
  \label{def:bridges}
Let $(\sigma,\alpha)$ be a collection of hypermaps.  
We say that a $3$-cycle $(i,j,k)$ of $\alpha$ is a {\em double bridge} if
$\kappa(\sigma,\alpha(i,k,j))=\kappa(\sigma,\alpha)+2$, it is a {\em
  single bridge} if
$\kappa(\sigma,\alpha(i,k,j))=\kappa(\sigma,\alpha)+1$, and it is an
{\em ordinary hyperedge} if
$\kappa(\sigma,\alpha(i,k,j))=\kappa(\sigma,\alpha)$.  
\end{definition}  
Clearly the properties defined in Definition~\ref{def:bridges} are also
mutually exclusive. They also cover all possibilities, as replacing
$(i,j,k)$ with $(i)(j)(k)$ increases the number of connected components
by at most $2$.

Not all $3\times 3$ possibilities created by pairing a property from
Definition~\ref{def:loops} with a property from
Definition~\ref{def:bridges} can be realized. A double loop is
necessarily ordinary, since $i,j$ and $k$ are already on the same orbit
of $\sigma$. Similarly a double bridge is necessarily regular: replacing
the hyperedge $(i,j,k)$ with $(i)(j)(k)$ can only increase the number of
connected components by two only if $i$, $j$ and $k$ belong to three
different vertices. These observations decrease the number of
possibilities by $3$. The remaining pairs are all possible, a small example
of each is shown in Table~\ref{tab:sample_types}.

\begin{table}[h]
\begin{center}
\begin{tabular}{|r||l|l|l|}
  \hline
  & regular & single loop & double loop\\
  \hline
 \hline 
ordinary&$\sigma=(1,4)(2,5,6)(3,7)$&$\sigma=(1,2,4)(3,5)$&$\sigma=(1,2,3)$\\
&$\alpha=(1,2,3)(4,5)(6,7)$&$\alpha=(1,2,3)(4,5)$&$\alpha=(1,2,3)$\\
\hline
  single bridge &$\sigma=(1,4)(2,5)(3)$&$\sigma=(1,2)(3)$&Does not\\
&$\alpha=(1,2,3)(4,5)$&$\alpha=(1,2,3)$&exist.\\
\hline
  double bridge&$\sigma=(1)(2)(3)$&Does not&Does not\\
&$\alpha=(1,2,3)$&exist&exist\\
  \hline
\end{tabular}
\end{center}
\caption{An example of all $6$ types of $3$-cycles of $\alpha$}
\label{tab:sample_types}
\end{table}

Using our observations we may also simplify our terminology, as shown in
Table~\ref{tab:types}. As observed above there is only one entry in the
row of a double bridge and in a column of a double loop. These terms
suffice by themselves. We will call a single bridge that is also a
single loop a {\em bridge-loop}, a truly ``exotic'' term for a graph
theorist. All other single bridges, respectively single loops will be
called {\em simple}. 

\begin{table}[h]
\begin{center}
\begin{tabular}{|r||c|c|c|}
  \hline
  & regular & single loop & double loop\\
  \hline
 \hline 
ordinary&ordinary regular&simple loop&double loop\\
\hline
  single bridge &simple bridge&bridge-loop& -- \\
\hline
  double bridge&double bridge& -- & --\\
  \hline
\end{tabular}
\end{center}
\caption{Simplified terminology}
\label{tab:types}
\end{table}

\section{Recurrences eliminating $3$-cycles}
\label{sec:rec}

In this section we apply Theorem~\ref{thm:wrgrecc} in the
special case when a collection of hypermaps $H=(\sigma,\alpha)$
contains a $3$-cycle. Together with Corollary~\ref{cor:wrg2rec}, the
results stated in this section allow us to recursively compute the
Whitney polynomial of any hypermap with short hyperedges. To alleviate
our notation, throughout this section we will use the shorthand
$R(\sigma,\alpha)$ for $R(\sigma,\alpha;u,v)$.   
Without loss of generality we may assume that our hypermap contains the 
cycle $(1,2,3)$. Furthermore, after a cyclic rotation of the indices,
if necessary, we may assume the following:
\begin{enumerate}
\item If $(1,2,3)$ is a simple or double loop then $1$ and $2$ belong to
  the same cycle of $(1,2,3)$.
\item Unless $(1,2,3)$ is a double bridge, deleting $(1,3)$ (to obtain
  $(1,2,3)(1,3)=(2,3)$) does not
  increase the number of connected components.
\end{enumerate}
We will maintain the above assumptions throughout this section without
repeating them in each statement. The main result of this section is the
following theorem. 
\begin{theorem}
\label{thm:reductions}  
 Consider a hypermap $H = (\sigma,\alpha)$ such that $(1,2,3)$ is a
 cycle of $\alpha$. Let $\alpha'=\alpha(1,3,2)$ be the permutation
 obtained by replacing the cycle $(1,2,3)$ with $(1)(2)(3)$ in $\alpha$. 
Then we
have the following recurrences. 
 \begin{enumerate}
\item If $(1,2,3)$ is a double loop then
$R(\sigma,\alpha)=(1+3v+v^2)\cdot R(\sigma,\alpha')$.
\item If $(1,2,3)$ is a double bridge then
$R(\sigma,\alpha)=(1+3u+u^2)\cdot R((1,2)(1,3)\sigma,\alpha')$. 
\item If $(1,2,3)$ is a simple loop where $1$ and $2$
belong to the same cycle of $\sigma$ then 
$$
R(\sigma,\alpha)=(1+v)\cdot
\left(R((2,3)\sigma,\alpha')+R(\sigma,\alpha')\right)
+R((1,3)\sigma,\alpha').
$$
\item If $(1,2,3)$ is a simple bridge where
$\kappa(\sigma,\alpha(2,3))=\kappa(\sigma,\alpha)+1$ then 
$$
R(\sigma,\alpha)=(1+u)\cdot \left(R((2,3)\sigma,\alpha')
+R((1,2)(1,3)\sigma,\alpha') 
  \right)+R((1,3)\sigma,\alpha').
$$
\item If $(1,2,3)$ is a bridge-loop where $1$ and $2$
belong to the same cycle of $\sigma$ then 
$$
R(\sigma,\alpha)=(1+u)\cdot(1+v)
\cdot R((2,3)\sigma,\alpha')+R((1,3)\sigma,\alpha').
$$
\item If $(1,2,3)$ is an ordinary regular hyperedge then 
\begin{align*}
R(\sigma,\alpha)&=R(\sigma,\alpha')+R((2,3)\sigma,\alpha')+R((1,2)\sigma,\alpha')\\
&+R((1,2)(1,3)\sigma,\alpha')+R((1,3)\sigma,\alpha').  
\end{align*}
 \end{enumerate}  
\end{theorem}  
We will prove Theorem~\ref{thm:reductions} in the rest of the section
through a sequence of lemmas and propositions. Theorems~\ref{thm:wrgrec} and
\ref{thm:wrgrecc} take the following form:  
\begin{align}
R(H)&=R(\phi_1(H))\cdot w_1+R(\phi_2(H))\cdot
w_2+R(\phi_3(H))\cdot w_3;\\
R(H)&=R(\psi_1(H))\cdot w_1+R(\psi_2(H))\cdot
w_2+R(\psi_3(H))\cdot w_3. 
\end{align}
Here $\phi_1(H)=(\sigma_1,\alpha_1)$ satisfies $\sigma_1=\sigma$ and
$\alpha_1$ is obtained from $\alpha$ by replacing $(1,2,3)$ with
$(1)(2,3)$. By our second assumption we obtain the following.
\begin{lemma}
\label{lem:psi1}
If $(1,2,3)$ is not a double bridge then we have $w_1=1$ and
$\psi_1(H)=(\sigma,\alpha(1,3))$. If $(1,2,3)$ is a double bridge then
$w_1=u$ and $\psi_1(H)=((1,2)\sigma,(1,2)\alpha)$ hold. 
\end{lemma}
Next we compute $\phi_2(H)=(\sigma_2,\alpha_2)$, $w_2$ and $\psi_2(H)$. By
definition $\alpha_2=\alpha_1$ holds. If $1$ and $2$ belong
to the same cycle of $\sigma$ then $w_2=v$ holds, in all other cases we
have have $w_2=1$. Finally $\psi_2(H)$ is obtained from $H$ by deleting
$(1,3)$ if $w_2=v$, in all other cases $\psi_2(H)$ is obtained from $H$
by contracting $(1,2)$. Keeping in mind our first assumption, we have
the following.
\begin{lemma}
  \label{lem:psi2}
If $(1,2,3)$ is a single or double loop then $w_2=v$ and
$\psi_2(H)=(\sigma,\alpha(1,3))$ hold. In all other cases we have
$w_2=1$ and $\psi_2(H)=((1,2)\sigma,(1,2)\alpha)$.   
\end{lemma}  
Finally, let us compute $\phi_3(H)=(\sigma_3,\alpha_3)$, $w_3$ and
$\psi_3(H)$. The permutation $\alpha_3$ is obtained from $\alpha$  by
replacing $(1,2,3)$ with $(1)(2)(3)$. We have 
$\sigma_3=\sigma$ only if $1$ and $3$ belong to the same cycle of
$\sigma$. By our first assumption, this may only happen when $(1,2,3)$
is a double loop. If $(1,2,3)$ is a double loop then $\sigma_3=\sigma$
and $w_3=v$ and $\psi_3(H)=\phi_3(H)$. In all other cases we have
$\sigma_3=(1,3)\sigma$.
\begin{lemma}
\label{lem:psi3} If $(1,2,3)$ is a double loop then we have $w_3=v$ and
$\psi_3(H)=\phi_3(H)=(\sigma,\alpha(1,3,2))$. In all other cases we have 
$\phi_3(H)=((1,3)\sigma,(1,3)\alpha(1,2))$,
$w_3=u^{\kappa(\phi_3(H))-\kappa(H)} $ and  
$$
\psi_3(H) =
\begin{cases}
((1,3)\sigma,(1,3)\alpha(1,2))& \mbox{if $\kappa(\phi_3(H))=\kappa(H)$},\\
((1,2)(1,3)\sigma,(1,2)(1,3)\alpha)& \mbox{if 
    $\kappa(\phi_3(H))=\kappa(H)+1$}.\\
\end{cases}  
$$
\end{lemma}  
Using the above lemmas we are ready to prove the six deletion-contraction
rules stated in Theorem~\ref{thm:reductions} in the sequence of the next six
propositions. 

\begin{proposition}
\label{prop:dloop}
If $(1,2,3)$ is a double loop then
$$
R(\sigma,\alpha)=(1+3v+v^2)\cdot R(\sigma,\alpha(1,3,2)).
$$
\end{proposition}
\begin{proof}
By Lemma~\ref{lem:psi1} we have $w_1=1$ and $\psi_1(H)=(\sigma,\alpha(1,3))$.  
Here $\psi_1(H)$ contains the
edge $(1,2,3)(1,3)=(2,3)$ which is a loop. Hence we obtain
$$
R(\psi_1(H))\cdot w_1=(1+v)\cdot R(\sigma,\alpha(1,3)(2,3)).
$$
By Lemma~\ref{lem:psi2} we have $w_2=v$ and $\psi_2(H)=(\sigma,\alpha(1,3))$.
Similarly to the previous computation we obtain 
$$
R(\psi_2(H))\cdot w_2=(1+v)v\cdot R(\sigma,\alpha(1,3)(2,3)).
$$
By Lemma~\ref{lem:psi3} we have $w_3=v$ and
$\psi_3(H)=\phi_3(H)=(\sigma,\alpha(1,3,2))$. Hence we obtain
$$
R(\psi_3(H))\cdot w_3=v\cdot R(\sigma,\alpha(1,3)(2,3)).
$$
The sum of the three display formulas in the proof yields the stated
equality. 
\end{proof}
The next statement is in a sense a dual of Proposition~\ref{prop:dloop} above.
\begin{proposition}
\label{prop:dbridge}
If $(1,2,3)$ is a double bridge then
$$
R(\sigma,\alpha)=(1+3u+u^2)\cdot R((1,3,2)\sigma,(1,3,2)\alpha).
$$
\end{proposition}  
\begin{proof}
By Lemma~\ref{lem:psi1} we have $w_1=u$ and
$\psi_1(H)=((1,2)\sigma,(1,2)\alpha)$.  Here $\psi_1(H)$ contains the
edge $(1,2)(1,2,3)=(2,3)$ which is a bridge. Hence we obtain
$$
R(\psi_1(H))\cdot w_1=(1+u)u\cdot
R((2,3)(1,2)\sigma,(2,3)(1,2)\alpha)).  
$$
By Lemma~\ref{lem:psi2} we have $w_2=1$ and
$\psi_2(H)=((1,2)\sigma,(1,2)\alpha)$. 
Similarly to the previous computation we obtain 
$$
R(\psi_2(H))\cdot w_2=(1+u)\cdot
R((2,3)(1,2)\sigma,(2,3)(1,2)\alpha)).  
$$
By Lemma~\ref{lem:psi3} we have
$\phi_3(H)=((1,3)\sigma,(1,3)\alpha(1,2))$,
$w_3=u$ and  
$$R(\psi_3(H))\cdot w_3
=u\cdot R((1,2)(1,3)\sigma,(1,2)(1,3)\alpha)).$$
The sum of the three display formulas in the proof yields the stated
equality. 
\end{proof}

\begin{proposition}
\label{prop:sloop}  
If $(1,2,3)$ is a simple loop where $1$ and $2$
belong to the same cycle of $\sigma$ then 
$$
R(\sigma,\alpha)=(1+v)\cdot
\left(R((2,3)\sigma,(2,3)\alpha(1,3))+R(\sigma,\alpha(1,3,2))\right)
+R((1,3)\sigma,(1,3)\alpha(1,2)).
$$
\end{proposition}  
\begin{proof}
By Lemma~\ref{lem:psi1} we have $w_1=1$ and $\psi_1(H)=(\sigma,\alpha(1,3))$.  
Here $\psi_1(H)$ contains the
edge $(1,2,3)(1,3)=(2,3)$ which is an ordinary edge. Hence we obtain
$$
R(\psi_1(H))\cdot
w_1=R(\sigma,\alpha(1,3)(2,3))+R((2,3)\sigma,(2,3)\alpha(1,3)).
$$
By Lemma~\ref{lem:psi2} we have $w_2=v$ and
$\psi_2(H)=(\sigma,\alpha(1,3))$. 
Similarly to the previous computation we obtain 
$$
R(\psi_2(H))\cdot
w_2=v\cdot \left(R(\sigma,\alpha(1,3)(2,3))+R((2,3)\sigma,(2,3)\alpha(1,3)))\right).
$$ 
By Lemma~\ref{lem:psi3} we have
$\phi_3(H)=((1,3)\sigma,(1,3)\alpha(1,2))$,
$w_3=1$ and  
$$R(\psi_3(H))\cdot w_3=R((1,3)\sigma,(1,3)\alpha(1,2)).$$
The sum of the three display formulas in the proof yields the stated
equality. 
\end{proof}
The dual of Proposition~\ref{prop:sloop} is the following. 
\begin{proposition}
\label{prop:sbridge}  
If $(1,2,3)$ is a simple bridge where
$\kappa(\sigma,\alpha(2,3))=\kappa(\sigma,\alpha)+1$ then 
\begin{align*}
  R(\sigma,\alpha)&=(1+u)\cdot
  \left(R((2,3)\sigma,(2,3)\alpha(1,3))
  +R((1,3,2)\sigma,(1,3,2)\alpha)
  \right)\\
  &+R((1,3)\sigma,(1,3)\alpha(1,2)).
\end{align*}  
\end{proposition}  
\begin{proof}
By Lemma~\ref{lem:psi1} we have $w_1=1$ and $\psi_1(H)=(\sigma,\alpha(1,3))$.  
Here $\psi_1(H)$ contains the
edge $(1,2,3)(1,3)=(2,3)$ which is a bridge. Hence we obtain
$$
R(\psi_1(H))\cdot w_1=(1+u)\cdot
R((2,3)\sigma,(2,3)\alpha(1,3))).  
$$
By Lemma~\ref{lem:psi2} we have $w_2=1$ and
$\psi_2(H)=((1,2)\sigma,(1,2)\alpha)$. Here $\psi_2(H)$ contains the
edge $(1,2)(1,2,3)=(2,3)$ which is a bridge. Hence we obtain
$$
R(\psi_2(H))\cdot w_1=(1+u)\cdot
R((2,3)(1,2)\sigma,(2,3)(1,2)\alpha)).  
$$
By Lemma~\ref{lem:psi3} we have
$\phi_3(H)=((1,3)\sigma,(1,3)\alpha(1,2))$,
$w_3=1$ and  
$$
R(\psi_3(H))\cdot w_3=R((1,3)\sigma,(1,3)\alpha(1,2)).
$$
The sum of the three display formulas in the proof yields the stated
equality. 
\end{proof}  
\begin{example}
\label{ex:sbridge}  
For $\sigma=(1,4)(2,5)(3)$ and $\alpha=(1,2,3)(4,5)$ the hypermap
$(\sigma,\alpha)$ is an example of a simple
bridge considered in Proposition~\ref{prop:sbridge} above. This hypermap
satisfies
\begin{align*}
R((2,3)\sigma,(2,3)\alpha(1,3))&=R((1,4)(2,5,3),(4,5))=u+1\\
R((1,3,2)\sigma,(1,3,2)\alpha)
&=R((1,4,3,2,5),(4,5))=v+1\quad\mbox{and}\\
R((1,3)\sigma,(1,3)\alpha(1,2))
&=R((1,4,3)(2,5),(4,5))=u+1.
\end{align*}
Proposition~\ref{prop:sbridge} gives
$$
R(\sigma,\alpha)=(1+u)\cdot (u+1+v+1)+u+1=u^2+uv+4u+v+3.
$$
The same formula, using Theorem~\ref{thm:wrgrec} directly, was obtained
in~\cite[Example~2.13]{Cori-Hetyei2}.
\end{example}  

\begin{proposition}
If $(1,2,3)$ is a bridge-loop where $1$ and $2$
belong to the same cycle of $\sigma$ then 
$$
R(\sigma,\alpha)=(1+u)\cdot(1+v)
\cdot R((2,3)\sigma,(2,3)\alpha(1,3))
+R((1,3)\sigma,(1,3)\alpha(1,2)).
$$
\end{proposition}  
\begin{proof}
By Lemma~\ref{lem:psi1} we have $w_1=1$ and $\psi_1(H)=(\sigma,\alpha(1,3))$.  
Here $\psi_1(H)$ contains the
edge $(1,2,3)(1,3)=(2,3)$ which is an bridge. Hence we obtain
$$
R(\psi_1(H))\cdot
w_1=(1+u)\cdot R((2,3)\sigma,(2,3)\alpha(1,3)).
$$
By Lemma~\ref{lem:psi2} we have $w_2=v$ and
$\psi_2(H)=(\sigma,\alpha(1,3))$. Similarly to the previous computation
we obtain
$$
R(\psi_2(H))\cdot
w_2=(1+u)v\cdot R((2,3)\sigma,(2,3)\alpha(1,3)).
$$
Lemma~\ref{lem:psi3} we have
$\phi_3(H)=((1,3)\sigma,(1,3)\alpha(1,2))$,
$w_3=1$ and  
$$R(\psi_3(H))\cdot w_3=R((1,3)\sigma,(1,3)\alpha(1,2)).$$
The sum of the three display formulas in the proof yields the stated
equality. 
\end{proof}

\begin{proposition}
\label{prop:or}
If $(1,2,3)$ is an ordinary regular hyperedge then 
\begin{align*}
  R(\sigma,\alpha)&=R(\sigma,\alpha(1,3,2))
       +R((2,3)\sigma,(2,3)\alpha(1,3))\\
       &+R((1,2)\sigma,(1,2)\alpha(2,3))+R((1,3,2)\sigma,(1,3,2)\alpha)\\
       &+R((1,3)\sigma,(1,3)\alpha(1,2)).
\end{align*}
\end{proposition}  
\begin{proof}
By Lemma~\ref{lem:psi1} we have $w_1=1$ and $\psi_1(H)=(\sigma,\alpha(1,3))$.  
Here $\psi_1(H)$ contains the edge $(1,2,3)(1,3)=(2,3)$ which is an
ordinary regular edge. Hence we obtain
$$
R(\psi_1(H))\cdot w_1=
R(\sigma,\alpha(1,3)(2,3))+
R((2,3)\sigma,(2,3)\alpha(1,3)).  
$$
By Lemma~\ref{lem:psi2} we have $w_2=1$ and
$\psi_2(H)=((1,2)\sigma,(1,2)\alpha)$. Here $\psi_2(H)$ contains the
edge $(1,2)(1,2,3)=(2,3)$ which is an ordinary regular edge.  
Hence we obtain
$$
R(\psi_2(H))\cdot w_2=
R((1,2)\sigma,(1,2)\alpha(2,3))+
R((2,3)(1,2)\sigma,(2,3)(1,2)\alpha).
$$
By Lemma~\ref{lem:psi3} we have
$\phi_3(H)=((1,3)\sigma,(1,3)\alpha(1,2))$,
$w_3=1$ and  
$$
R(\psi_3(H))\cdot w_3=R((1,3)\sigma,(1,3)\alpha(1,2)).
$$
The sum of the three display formulas in the proof yields the stated
equality.

\end{proof}  
An important consequence of the above statements is the following.
\begin{corollary}
\label{cor:almostTutte}  
Let $(\sigma,\alpha)$ be a collection of hypermaps with short hyperedges. 
If we are able to use the recurrences stated above together with
Corollary~\ref{cor:wrg2rec} to compute  $R(\sigma,\alpha)$ in such a
way that we never encounter a double loop or a double bridge in the
process, then $R(\sigma,\alpha)$ is a nonnegative integer linear
combination of products of powers of $(1+u)$ and $(1+v)$.
\end{corollary}  
In other words, if we are able to avoid double bridges and double loops
in our computation of the Whitney polynomial, then after substituting
$u=x-1$ and $y=v-1$ we obtain a 
polynomial of $x$ and $y$ that ``looks like a Tutte polynomial'': it has
nonnegative integer coefficients.

\section{Substitutions into the Whitney polynomial}
\label{sec:subs}

A few evaluations of the Whitney polynomial of a collection of hypermaps
are listed in~\cite{Cori-Hetyei2}. In this section
we add a few results that are specific to collections of hypermaps with
short hyperedges. 

\begin{remark}
It is worth noting that Proposition~6.10 and
Theorem~7.7 in~\cite{Cori-Hetyei2} are specifically about collections of
hypermaps with short edges, showing that the numbers of proper colorings
and nowhere zero flows (as defined in~\cite{Cori-Machi-flow}) are
polynomial functions albeit they can not be obtained by direct
substitutions into our Whitney polynomial. 
\end{remark}

It is well-known that substituting $x=0$ and $y=0$ into the Tutte
polynomial of a graph with at least one edge yields zero. Equivalently,
the Whitney polynomial $R(\sigma,\alpha;u,v)$ of a map with at least one
edge that is not a bud satisfies $R(\sigma,\alpha;-1,-1)=0$. This
observation may be extended to arbitrary hypermaps as follows.

\begin{proposition}
\label{prop:edgezero}  
Let $(\sigma,\alpha)$ be any hypermap on the set $\{1,2,\ldots,n\}$. If
$\alpha$ contains at least one cycle of length $2$ then we have
$R(\sigma,\alpha;-1,-1)=0$. 
\end{proposition}  
\begin{proof}
We proceed by induction on $n-z(\alpha)$. Since $\alpha$ has at least
one $2$-cycle, we have $z(\alpha)\leq n-1$ and $n-z(\alpha)\geq 1$. 

If $n-z(\alpha)=1$ then $\alpha$ has exactly one two-cycle $(i,j)$, and
all elements of $\{1,2,\ldots,n\}-\{i,j\}$ are fixed points of $\alpha$.  
The hypermap $(\sigma,\alpha)$ is a map and has at most two
vertices. The edge $(i,j)$ is a bridge if $z(\sigma)=2$ and it is a loop
if $z(\sigma)=1$. By Corollary~\ref{cor:wrg2rec}, the Whitney polynomial
$R(\sigma,\alpha;u,v)$ is either $u+1$ or $v+1$. Substituting $u=-1$ and
$v=-1$ yields $R(\sigma,\alpha;-1,-1)=0$ in either case.

If $\alpha$ has at least two cycles of length at least $2$, let us
select one of its longest cycles and apply Theorem~\ref{thm:wrgrecc} to
it. We obtain a formula of the form
$$
R(\sigma,\alpha;u,v)=\sum_{k=1}^m R(\sigma_k,\alpha_k;u,v)\cdot w_k,
$$
where each hypermap $(\sigma_k,\alpha_k)$ satisfies
$z(\alpha_k)>z(\alpha)$ and hence
$n-z(\alpha_k)<n-z(\alpha)$. Furthermore each $\alpha_k$ contains at
least one cycle of length $2$. Substituting $u=-1$ and $v=-1$ yields
zero on the right hand side by our induction hypothesis.   
\end{proof}
Proposition~\ref{prop:edgezero} suggests having a closer look at {\em
  $3$-uniform hypermaps}.
\begin{definition}
We say that a hypermap $(\sigma,\alpha)$ with short hyperedges is {\em
  $3$-uniform} if $\alpha$ has no cycle of length $2$. Equivalently,
each cycle of $\alpha$ is either a bud, or has length $3$.  
\end{definition}

\begin{remark}
It is worth noting that the Whitney polynomial of a $3$-uniform
  hypermap may be computed using only the results stated in
  Section~\ref{sec:rec}: each statement expresses the Whitney polynomial
  of a $3$-uniform hypermap in terms of the Whitney polynomials of other
  $3$-uniform hypermaps.  
\end{remark}

While evaluating $R(\sigma,\alpha;-1,-1)$ for $3$-uniform hypermaps, the
following observation plays a key role.

\begin{lemma}
For a $3$-uniform hypermap $(\sigma,\alpha)$ the number $z(\sigma)$ of
vertices and the number $z(\alpha^{-1}\sigma)$ faces have the same parity.
\end{lemma}
\begin{proof}
Introducing $z_{i}(\alpha)$ for the number of $i$-cycles of $\alpha$,
the number $n$ of points is given by $n=z_1(\alpha)+3\cdot
z_3(\alpha)$. Substituting this equality and $z(\alpha)=z_1(\alpha)+z_3(\alpha)$
into equation (\ref{eq:genusdef}) we obtain
$$
z_1(\alpha)+3\cdot
z_3(\alpha) + 2 -2g(\sigma,\alpha) = z(\sigma) + z_1(\alpha)+z_3(\alpha)
+ z(\alpha^{-1}\sigma). 
$$
After simplifying we obtain
$$
2\cdot
\left(z_3(\alpha) + 1 -g(\sigma,\alpha)\right)
= z(\sigma) + z(\alpha^{-1}\sigma).  
$$

\end{proof}  

\begin{theorem}
  A $3$-uniform hypermap $(\sigma,\alpha)$ satisfies
  $$
  R(\sigma,\alpha;-1,-1)=(-1)^{\frac{z(\sigma)-z(\alpha^{-1}\sigma)}{2}}
  $$
and $R(\sigma,\alpha;-2,-1)=R(\sigma,\alpha;-1,-2)=R(\sigma,\alpha;-1,-1)$.   
\end{theorem}  
\begin{proof}
We provide the details for the computation of $R(\sigma,\alpha;-1,-1)$
only, the computations of $R(\sigma,\alpha;-2,-1)$ and
$R(\sigma,\alpha;-1,-2)$ are completely analogous. We proceed by
induction on the number of $3$-cycles of $\alpha$. If $\alpha$ has no
$3$-cycle then it is the identity permutation and $\sigma$ must be
circular: we must have $z(\sigma)=z(\alpha^{-1}\sigma)=1$ and the stated
formula gives $R(\sigma,\alpha;-1,-1)=1$, which is the correct value.

If $\alpha$ contains at least one $3$-cycle, without loss of generality
we may assume that $(1,2,3)$ is a cycle of $\alpha$, and it satisfies the
conditions stated at the beginning of Section~\ref{sec:rec}. We may
eliminate the cycle $(1,2,3)$ using one of the recurrences stated in
Theorem~\ref{thm:reductions}. After setting $u=-1$ and
$v=-1$ the table becomes even more sparse: we may replace all entries
$1+u$, $1+v$ and the entry $(1+u)(1+v)$ with zero. Furthermore, we may
replace $u^2+3u+1$ and $v^2+3v+1$ with $-1$. The induction step may now
be shown on a case by case basis, using Table~\ref{tab:vf} below. Each
hypercontraction $\gamma_{(i,j)}$ 
decreases the number of vertices by one and leaves the number of faces
unchanged, and each hyperdeletion $\delta_{(i,j)}$ leaves the number of
vertices unchanged and decreases the number of faces by one. Hence the
difference between the number of vertices and the number of faces is the
same as in $H=(\sigma,\alpha)$ for all hypermaps on the right hand side
except for the hypermaps $\gamma_{(1,3,2)}H$, and $\delta_{(1,3,2)}H$. 
By our induction hypothesis, the evaluation of the Whitney polynomial at
$u=v=-1$ is $(-1)^{\frac{z(\sigma)-z(\alpha^{-1}\sigma)}{2}}$ for all
hypermaps appearing on the right hand sides of our recurrences except
for $\gamma_{(1,3,2)}H$, and $\delta_{(1,3,2)}H$ whose Whitney
polynomial evaluated at $u=v=-1$ yields
$(-1)^{1+\frac{z(\sigma)-z(\alpha^{-1}\sigma)}{2}}$. If $(1,2,3)$ is a
double loop or a double bridge, we obtain $R(\sigma,\alpha;-1,-1)$ by
multiplying $(-1)^{1+\frac{z(\sigma)-z(\alpha^{-1}\sigma)}{2}}$ with
$(-1)^2+3\cdot (-1)+1=-1$. If $(1,2,3)$ is a simple loop, a simple
bridge or a double loop, we obtain $R(\sigma,\alpha;-1,-1)$ by
multiplying $(-1)^{\frac{z(\sigma)-z(\alpha^{-1}\sigma)}{2}}$ with
$1$. Finally, if $(1,2,3)$ is an ordinary regular edge then
$R(\sigma,\alpha;-1,-1)$ is a sum of five terms, three of which is
$(-1)^{\frac{z(\sigma)-z(\alpha^{-1}\sigma)}{2}}$ and the two others are
$(-1)^{1+\frac{z(\sigma)-z(\alpha^{-1}\sigma)}{2}}$. 
\begin{table}[h]
\begin{tabular}{l|ccccc|}
    &$\gamma_{(1,3,2)}H$
    &$\gamma_{(1,2)}\delta_{(2,3)}H$
      &$\gamma_{(2,3)}\delta_{(1,3)}H$
        &$\gamma_{(1,3)}\delta_{(1,2)}H$
  &$\delta_{(1,3,2)}H$\\
  \hline
  \hline
  vertices  &$-2$&$-1$&$-1$&$-1$&$0$\\
\hline  
faces  & $0$&$-1$&$-1$&$-1$& $-2$\\
\hline
\end{tabular}  
\caption{Change in the numbers of vertices and faces after replacing $H$
  with the hypermap labeling the column.}
\label{tab:vf}
\end{table}
We obtain the stated equality in all six cases. 
\end{proof}  

\section{The reliability polynomial and the random cluster model}
\label{sec:rel}

In this section we generalize the reliability polynomial and the random
cluster model from graphs and maps to hypermaps. Our construction works
for hypermaps in general, but its specialization to hypermaps with short
hyperedges is particularly interesting as the results only depend on the
underlying (multi)hypergraph structure and not the topological
information. 

Given a connected graph $G$ on the vertex set $V$ and edge set $E$,
the {\em reliability polynomial}
$${\mathcal R}(G;p)=\sum_{A} p^{|A|} (1-p)^{|E-A|}$$
expresses the probability of the graph remaining
connected after an accident where each edge survives with probability
$p$ and gets destroyed with probability $1-p$, independently. Here the
sum is over all spanning subsets $A$ of $E$. It is well known
(see~\cite[Equation~(3.3)]{Welsh}) that ${\mathcal R}(G;p)$ may be
expressed in terms of the Tutte polynomial $T(G;x,y)$ of the graph as
follows:
\begin{equation}
\label{eq:relp}
  {\mathcal R}(G;p)= q^{|E|-|V|+1} p^{|V|-1} T(G;1,q^{-1}),
\end{equation}  
where $q=1-p$. First we introduce a probabilistic model that
allows generalizing~\eqref{eq:relp} from maps to hypermaps. 

Let $(\sigma,\alpha)$ be a hypermap on the set $\{1,\ldots,n\}$, and
let $t\geq 0$ be a fixed parameter. After an accident, each cycle of
$\alpha$ is independently replaced by a random refinement, resulting in
a random refinement $\beta$ of $\alpha$. The description of the random
event is the same for each cycle $(i_1,\ldots,i_m)$ of $\alpha$. The
refinements of $(i_1,\ldots,i_m)$ 
are precisely the noncrossing partitions on the set $\{i_1,\ldots,i_m\}$
with respect to the circular order $(i_1,\ldots,i_m)$. The number of noncrossing
partitions of an $m$ element set into $k$ parts is the {\em Narayana
  number}
$$
N(m,k)=\frac{1}{m} \binom{m}{k}\binom{m}{k-1},
$$
see~\cite[(1)]{Simion-noncrossing}. The {\em Narayana polynomial
  $N_m(t)$} is defined as 
$$
N_m(t)=\sum_{k=0}^m N(m,k) t^k.
$$
We will only consider these polynomials for positive values of $m$,
when the constant term of $N_m(t)$ is zero. Hence we may define the {\em
  reduced Narayana polynomials} as
$$
\widetilde{N}_m(t)=\sum_{k=1}^m N(m,k) t^{k-1}.
$$
We define the probability of a given noncrosssing partition into $k$
parts replacing $(i_1,\ldots,i_m)$ to be
$t^{k-1}/\widetilde{N}_m(t)$. It is 
immediate from the definitions that the sum of the probabilities of all
refinements of $(i_1,\ldots,i_m)$ is $1$. Furthermore, for $t>0$ the
probability $t^{k-1}/\widetilde{N}_m(t)$ is the same as
$t^{k}/N_m(t)$. For $m=2$ we obtain that the probabilities of the
survival, respectively destruction of an edge are
\begin{equation}
  \label{eq:pq}
p=\frac{1}{1+t}\quad\mbox{respectively}\quad q=\frac{t}{1+t}.
\end{equation}  
The probability that we obtain a hypermap $(\sigma,\beta)$ after our
random accident is 
\begin{equation}
\label{eq:relp1}  
{\mathcal R}(\sigma,\alpha;t)=\prod_{i=1}^n
\frac{1}{\widetilde{N}_i(t)^{z_i(\alpha)}} \sum_{\beta}
t^{z(\beta)-z(\alpha)}.
\end{equation}
Here $z_i(\alpha)$ is the number of cycles of $\alpha$ of length $i$ and 
the summation is over all spanning refinements $\beta$, that is, over all
refinements satisfying $\kappa(\sigma,\beta)=1$. 
\begin{definition}
We define the function ${\mathcal R}(\sigma,\alpha;t)$ given in
\eqref{eq:relp1} as the {\em reliability function} of the hypermap
$(\sigma,\alpha)$. 
\end{definition}
\begin{remark}
The reliability function is a rational function but not
necessarily polynomial of $t$, even in the case when $(\sigma,\alpha)$
is a map. The reliability function of a map is obtained from its
reliability polynomial using~\eqref{eq:pq}. See~\eqref{eq:relpm}
below.
\end{remark}
After multiplying both sides with 
$t^{z(\alpha)}$ we may rewrite~\eqref{eq:relp1} as 
\begin{equation}
\label{eq:relp2}  
{\mathcal R}(\sigma,\alpha;t)=\prod_{i=1}^n
\frac{1}{N_i(t)^{z_i(\alpha)}} \sum_{\beta}
t^{z(\beta)}.
\end{equation}
This rational function yields the correct probabilities for any $t>0$, but
for $t=0$ we need to take its right limit at $t\rightarrow 0^+$ to
obtain the correct probability.
Observe next that evaluating the Whitney polynomial
$R(\sigma,\alpha;u,v)$ at $u=0$ and $v=t^{-1}$ yields
$$
R(\sigma,\alpha;0,t^{-1})=\sum_{\beta} t^{-n-1+z(\beta)+z(\sigma)}
$$
where the summation is over all spanning refinements $\beta$ of
$\alpha$. Indeed, substituting $u=0$ into the factor
$u^{\kappa(\sigma,\beta)-\kappa(\sigma,\alpha)}$ yields a factor of $1$ if
$\kappa(\sigma,\beta)=\kappa(\sigma,\alpha)=1$ and it yields a factor of
$0$ in all other cases. Comparing the last display formula
with~\eqref{eq:relp1} and \eqref{eq:relp2} we obtain
\begin{equation}
\label{eq:relp3}  
{\mathcal R}(\sigma,\alpha;t)
      = t^{n+1-z(\sigma)-z(\alpha)}\prod_{i=1}^n
\frac{1}{\widetilde{N}_i(t)^{z_i(\alpha)}}
R(\sigma,\alpha;0,t^{-1})\quad\mbox{and}
\end{equation}
\begin{equation}
\label{eq:relp4}  
{\mathcal R}(\sigma,\alpha;t)
      = t^{n+1-z(\sigma)}\prod_{i=1}^n
\frac{1}{N_i(t)^{z_i(\alpha)}} R(\sigma,\alpha;0,t^{-1}).
\end{equation}
For maps, we have $z_1(\alpha)+z_2(\alpha)=z(\alpha)$ and
$z_1(\alpha)+2z_2(\alpha)=n$. Hence the number of nontrivial
edges is $z_2(\alpha)=n-z(\alpha)$. Keeping in mind
$\widetilde{N}_2(t)=1+t$, Equation~\eqref{eq:relp3}
may be rewritten as \begin{equation}
\label{eq:relpm}  
{\mathcal R}(\sigma,\alpha;t)
      = \frac{t^{z_2(\alpha)-z(\sigma)+1}}{(1+t)^{z_2(\alpha)}}
R(\sigma,\alpha;0,t^{-1}).
\end{equation}
This equation is easily seen to be equivalent to \eqref{eq:relp},
using the substitution rules \eqref{eq:pq}.

As for graphs, we may extend our hypermap model to a generalization of the {\em
  random cluster model} introduced by Fortuin and
Kasteleyn~\cite{Fortuin-Kasteleyn}. We will extend the notation used by
Welsh~\cite[Section 4]{Welsh}. Let $(\sigma,\alpha)$ be a collection of
hypermaps  on
the set of points $\{1,\ldots,n\}$, and
let $\beta$ be a refinement of $\alpha$. We redefine the probability of
$\beta$ to be
\begin{equation}
  \label{eq:mubeta}
\mu(\beta) = Z^{-1}\cdot Q^{\kappa(\sigma,\beta)}
\frac{t^{z(\beta)-z(\alpha)}}{\prod_{i=1}^n \widetilde{N}_i(t)^{z_i(\alpha)}}.  
\end{equation}   
Here $Q$ and $t$ are nonnegative parameters and $Z$ is the normalizing
factor defined by $\sum_{\beta\leq \alpha} \mu(\beta)=1$. In other
words, $Z$ is the {\em partition function} of the model, given by
\begin{equation}
  \label{eq:Z}
Z(\sigma,\alpha;Q,t)=\sum_{\beta\leq \alpha} Q^{\kappa(\sigma,\beta)}
\frac{t^{z(\beta)-z(\alpha)}}{\prod_{i=1}^n \widetilde{N}_i(t)^{z_i(\alpha)}}. 
 \end{equation} 
Observe next that evaluating the Whitney polynomial
$R(\sigma,\alpha;u,v)$ at $u=Qt$ and $v=t^{-1}$ yields
$$
R(\sigma,\alpha;Qt,t^{-1})=\sum_{\beta\leq \alpha}
Q^{\kappa(\sigma,\beta)-\kappa(\sigma,\alpha)}
t^{-n+z(\beta)+z(\sigma)-\kappa(\sigma,\alpha)}. 
$$
Comparing this equation with \eqref{eq:Z} we obtain
\begin{equation}
  \label{eq:Zh}
Z(\sigma,\alpha;Q,t)=\frac{Q^{\kappa(\sigma,\alpha)}
t^{n+\kappa(\sigma,\alpha)-z(\sigma)-z(\alpha)}}{\prod_{i=1}^n
  \widetilde{N}_i(t)^{z_i(\alpha)}}\cdot  R(\sigma,\alpha;Qt,t^{-1}). 
\end{equation}  
For maps, \eqref{eq:Zh} may be rewritten as 
\begin{equation}
  \label{eq:Zm}
Z(\sigma,\alpha;Q,t)=\frac{Q^{\kappa(\sigma,\alpha)}
t^{z_2(\alpha)+\kappa(\sigma,\alpha)-z(\sigma)}}{
  (1+t)^{z_2(\alpha)}}\cdot  R(\sigma,\alpha;Qt,t^{-1}), 
\end{equation}  
which, using the substitution rules~\eqref{eq:pq}, is  easily seen to
be equivalent to
$$
Z(\sigma,\alpha;Q,t)=p^{z(\sigma)-\kappa(\sigma,\alpha)}q^{z_2(\alpha)-z(\sigma)+\kappa(\sigma,\alpha)}\cdot
Q^{\kappa(\sigma,\alpha)}\cdot 
R\left(\sigma,\alpha;\frac{Qq}{p},q^{-1}-1\right).  
$$
This is~\cite[(4.1)]{Welsh} in our notation.

\section{Associated two-colored maps}
\label{sec:a2cm}

In the rest of this manuscript we develop formulas for counting the
number of spanning hyperforests in a collection of hypermaps with short
hyperedges. By~\cite[Proposition~2.5]{Cori-Hetyei2} this number is the
value of the Whitney polynomial evaluated at $u=v=0$. A key tool in our
computation will be to associate a collection of maps with a
$2$-coloring on the vertices, as follows. 

\begin{definition}
\label{def:2cm}
Let $(\sigma,\alpha)$ be a collection of hypermaps with short hyperedges 
on the set of points $\{1,2,\ldots,n\}$. We define its {\em
  associated collection of
two-colored maps $(\Ss{\alpha}(\sigma),\Aa{\alpha}(\alpha))$} as follows: 
\begin{enumerate}
\item We keep each cycle of $\sigma$ as a cycle of $\Ss{\alpha}(\sigma)$. 
\item We keep each $2$-cycle $(i,j)$ of $\alpha$ as a cycle of $\Aa{\alpha}(\alpha)$.
\item For each each $3$-cycle $(i,j,k)$ of $\alpha$ we add the cycle 
  $(i',k',j')$ to $\Ss{\alpha}(\sigma)$ and the set of edges $\{(i,i'), (j,j'), (k,k')\}$ 
to $\Aa{\alpha}(\alpha)$. Here each point $i'$ is different from the set of points
contained in $\{1,2,\ldots,n\}$ and $i'\neq j'$ holds whenever $i\neq j$.
\end{enumerate}  
We call the originally present points of the set $\{1,2,\ldots,n\}$ {\em
  black points} and the newly added points $i'$ added
by rule~(3) {\em white points}. We color each vertex in
$\Ss{\alpha}(\sigma)$ with the common color of all of its points.
\end{definition}

\begin{remark}
Our definition is somewhat reminiscent to the one presented by
Walsh~\cite{Walsh} who transforms every hypermap $(\sigma,\alpha)$ into
a bipartite map by replacing all hyperedges with vertices and creating edges
between the original vertices and the vertices representing the
hyperedges. The main difference is that we don't transform the
$2$-cycles of $\alpha$ into vertices, hence our
associated collection of two-colored maps is not bipartite when $\alpha$
has a $2$-cycle. 
\end{remark}   

\begin{example}
\label{ex:hypermap-g}
The hypermap $(\sigma,\alpha)$ given by
$$\sigma=(1,4,8)(2,5)(3,6)(7)\quad\mbox{and}\quad
\alpha=(1,2,3)(4,5)(6,7,8)$$
is shown on the left hand side of
Fig.~\ref{fig:hypermap}. The associated two-colored hypermap,
given by
\begin{align*}\Ss{\alpha}(\sigma)&=(1,4,8)(2,5)(3,6)(7)(1',3',2')(6',8',7')\quad\mbox{and}\\
  \Aa{\alpha}(\alpha)&=(4,5)(1,1')(2,2')(3,3')(6,6')(7,7')(8,8'),
\end{align*}  
is shown on the right hand side of the same figure. 
\end{example}

\begin{figure}[h]
%100%  
\begin{center}
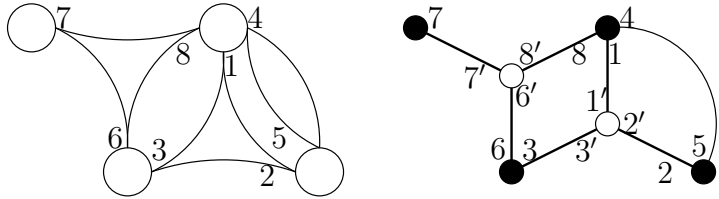
\end{center}
\caption{A hypermap with short hyperedges and its associated two-colored
map}
\label{fig:hypermap} 
\end{figure}

If $(\sigma,\alpha)$ is a hypermap with short hyperedges then the
associated collection of two-colored maps $(\Ss{\alpha}(\sigma),\Aa{\alpha}(\alpha))$ is a
map. More generally we have the following result. 

\begin{proposition}
\label{prop:samekappa}
Let $(\sigma,\alpha)$ be a collection of hypermaps with short
hyperedges. Then then the associated collection of maps
satisfies
$\kappa(\sigma,\alpha)=
\kappa(\Ss{\alpha}(\sigma),\Aa{\alpha}(\alpha))$. 
\end{proposition}
\begin{proof}
Using the edges $(i,i')$ we can see that each newly added
white point $i'$ is on the same orbit of the permutation group $\langle
\Ss{\alpha}(\sigma),\Aa{\alpha}(\alpha)\rangle$ generated by $\Ss{\alpha}(\sigma)$ and $\Aa{\alpha}(\alpha)$ as some
black point $i$. Two black points $i$ and $j$
belong to the same cycle of $\sigma$ if and only if they belong to the
same cycle of $\Ss{\alpha}(\sigma)$. Assume next that $i$ and $j$ belong to the same
cycle of $\alpha$. If this cycle is a $2$-cycle, then $(i,j)$ is also a
cycle of $\Aa{\alpha}(\alpha)$. If this cycle is a cycle of the form $(i,j,k)$ then
$i$ and $j$ are on the same orbit of $\langle \Ss{\alpha}(\sigma),\Aa{\alpha}(\alpha)\rangle$ via
$$
i\xrightarrow{\Aa{\alpha}(\alpha)} i'\xrightarrow{\Ss{\alpha}(\sigma)^2} j'
\xrightarrow{\Aa{\alpha}(\alpha)} j.
$$
It may be shown similarly that  $i$ and $j$ are on the same $\langle
\Ss{\alpha}(\sigma),\Aa{\alpha}(\alpha)\rangle$-orbit if $(i,k,j)$ is a cycle of $\alpha$ for
some $k$.
Assume finally that two black points are $i$ and $j$ are on
the same orbit of $\langle \Ss{\alpha}(\sigma),\Aa{\alpha}(\alpha)\rangle$ and consider a
sequence of operations $\tau_m\tau_{m-1}\cdots \tau_1$, where each
$\tau_{k}$ belongs to $\{\Ss{\alpha}(\sigma),\Aa{\alpha}(\alpha)\}$, taking $i$ into $j$:
$$
i\xrightarrow{\tau_1} i_1 \xrightarrow{\tau_2} i_2 \xrightarrow{\tau_3} \cdots
\xrightarrow{\tau_m} i_m=j.  
$$
If some $\tau_k$ takes an black point $i_{k-1}$ into an
black point $i_k$ then we either have $i_k=\sigma(i_{k-1})$
and $\tau_k$ may be replaced by a copy of $\sigma$, or $(i_{k-1},i_k)$ is
a $2$-cycle of $\alpha$, hence $\tau_k$ may be replaced by a copy of
$\alpha$.  If some $\tau_k$ takes an black point $i_{k-1}$
into a white point $i_k$ then there is a least $\ell>k$ such that 
$i_{\ell}$ is once again an black point. The labeled
directed path
$$
i_{k-1}\xrightarrow{\tau_k} i_k \xrightarrow{\tau_{k+1}} \cdots
\xrightarrow{\tau_{\ell}} i_{\ell}  
$$
must be of the form
$$
i_{k-1}\xrightarrow{\Aa{\alpha}(\alpha)} i_{k-1'} \xrightarrow{\Ss{\alpha}(\sigma)} \cdots
\xrightarrow{\Ss{\alpha}(\sigma)} i_{\ell-1}' \xrightarrow{\Aa{\alpha}(\alpha)} i_{\ell}  
$$
and we may replace this entire segment either by $i_{k-1}\xrightarrow{\alpha}
i_{\ell}$ or by $i_{k-1}\xrightarrow{\alpha^2} i_{\ell}$. 
\end{proof}

The details of the proof of the next statement are similar to that of
Proposition~\ref{prop:samekappa} and omitted.

\begin{proposition}
\label{prop:samefaces}
Let $(\sigma,\alpha)$ be a collection of hypermaps  with short hyperedges and
let $(\Ss{\alpha}(\sigma),\Aa{\alpha}(\alpha))$ be its associated two-colored collection of
maps. Then the cycles of $\Aa{\alpha}(\alpha)^{-1}\Ss{\alpha}(\sigma)$ are obtained form the
cycles of $\alpha^{-1}\sigma$ by inserting each white point $i'$
between $\sigma^{-1} i$ and $(\alpha^{-1}\sigma)\sigma^{-1} i$ in the
cycle of $\alpha^{-1}\sigma$ containing $i$.
\end{proposition}
Indeed, a white point $i'$ is taken by $\Ss{\alpha}(\sigma)$ into
$\alpha^2(i)'$ and this is taken by $\Aa{\alpha}(\alpha)^{-1}=\Aa{\alpha}(\alpha)$ into
$\alpha^2(i)$. Similarly, $\Ss{\alpha}(\sigma)$ takes $\sigma^{-1}(i)$ into $i$ and
this one is taken by $\Aa{\alpha}(\alpha)^{-1}=\Aa{\alpha}(\alpha)$ into $i'$. We obtain
$$
\sigma^{-1} i\xrightarrow{\Aa{\alpha}(\alpha)^{-1}\Ss{\alpha}(\sigma)} i'
\xrightarrow{\Aa{\alpha}(\alpha)^{-1}\Ss{\alpha}(\sigma)} \alpha^2(i)=(\alpha^{-1}\sigma)\sigma^{-1} i. 
$$
\begin{corollary}
\label{cor:samefaces}  
The collection of hypermaps $(\sigma,\alpha)$ and its associated
collection of two-colored maps have the same number
of faces:  
$z(\alpha^{-1}\sigma)=z(\Aa{\alpha}(\alpha)^{-1}\Ss{\alpha}(\sigma))$ holds. 
\end{corollary}
Using Proposition~\ref{prop:samekappa} and Corollary~\ref{cor:samefaces} it is
easy to show that the associated collection of two-colored maps also has
the same genus:
\begin{proposition}
A collection of hypermaps $(\sigma,\alpha)$ with short hyperedges
satisfies $g(\sigma,\alpha)=g(\Ss{\alpha}(\sigma),\Aa{\alpha}(\alpha))$.
\end{proposition}
\begin{proof}  
Let us denote the number of $3$-cycles of $\alpha$ by
$z_3(\alpha)$. Using this notation
\begin{align}
  z(\Ss{\alpha}(\sigma))&=z(\sigma)+z_3(\alpha)\quad\mbox{and}\\
z(\Aa{\alpha}(\alpha))&=z(\alpha)+2z_3(\alpha)
\end{align}
are direct consequences of the definitions. Using these equations, the
statement is a direct consequence of the genus formula and
of Propositions~\ref{prop:samekappa} and \ref{prop:samefaces}.   
\end{proof}

We conclude this section with the description of an operation that
represents each refinement $\beta$ of $\alpha$  with a refinement
$\Rr{\alpha}(\beta)$ of $\Aa{\alpha}(\alpha)$ in the map $(\Ss{\alpha}(\sigma),\Aa{\alpha}(\alpha))$.
\begin{definition}
Let $(\sigma,\alpha)$ be a collection of hypermaps with short hyperedges and let
$(\Ss{\alpha}(\sigma),\Aa{\alpha}(\alpha))$ be the
associated collection of two-colored maps. To each
refinement $\beta$ of $\alpha$ we associate a refinement $\Rr{\alpha}(\beta)$ of
$\Aa{\alpha}(\alpha)$ by applying the following modifications to the
cycles of $\Aa{\alpha}(\alpha)$:
\begin{enumerate}
\item For each $2$-cycle $(i,j)$ of $\alpha$ which is replaced by the pair
  of fixpoints $(i)(j)$ in $\beta$ we also replace the $2$-cycle $(i,j)$
  of $\Aa{\alpha}(\alpha)$ with $(i)(j)$ in $\Rr{\alpha}(\beta)$. 
\item For each $3$-cycle $(i,j,k)$ of $\alpha$ which is replaced by 
  $(i,j)(k)$ in $\theta$ we replace the edge $(k,k')$ of $\Aa{\alpha}(\alpha)$ with
  $(k)(k')$ in $\Rr{\alpha}(\beta)$. 
\item For each $3$-cycle $(i,j,k)$ of $\alpha$ which is replaced by the
  triplet of fixpoints $(i)(j)(k)$ in $\theta$ we replace the edges $(i,i')$,
  $(j,j')$ and $(k,k')$ of $\Aa{\alpha}(\alpha)$ with $(i)(i')$,
  $(j)(j')$ and $(k)(k')$ in $\Rr{\alpha}(\beta)$.  
\end{enumerate}  
\end{definition}  
\begin{figure}[h]
%100%  
\begin{center}
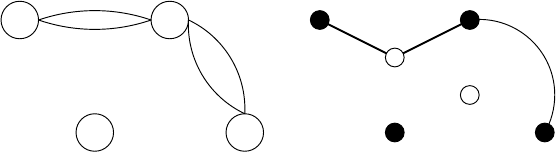
\end{center}
\caption{The collection of maps $(\Ss{\alpha}(\sigma),\Rr{\alpha}(\beta))$ for $\beta=(4,5)(7,8)$}
\label{fig:refinement} 
\end{figure}
Consider the hypermap $(\sigma,\alpha)$ introduced in
Example~\ref{ex:hypermap-g} and the refinement $\beta=(4,5)(7,8)$ of
$\alpha$. The cycle $(6,7,8)$ of $\alpha$ is replaced by $(7,8)$ in
$\beta$ and $(3,6)$ becomes an isolated vertex of
$(\sigma,\beta)$. Correspondingly, we delete the edge $(6,6')$ in
$(\Ss{\alpha}(\sigma),\Aa{\alpha}(\alpha))$ and this turns the black vertex $(3,6)$ also into an
isolated point. The collection of maps $(\Ss{\alpha}(\sigma),\Rr{\alpha}(\beta))$ also contains
the additional isolated white vertex $(1',3',2')$. No connected component of
$(\sigma,\beta)$ corresponds to such an isolated white vertex, because
it contains only white points. 
\begin{definition}
Let $(\sigma,\alpha)$ be a collection of hypermaps with short hyperedges
and let $\beta$ be a refinement of $\alpha$. We call a cycle of
$\Ss{\alpha}(\sigma)$ a {\em singularity of $\beta$} if it
represents a $3$-cycle $(i,j,k)$ of $\alpha$ that is replaced with
$(i)(j)(k)$ in $\beta$. We denote the set of singularities of $\beta$ by
$S(\beta)$.    
\end{definition}
The following key result relates the collection of maps
$(\Ss{\alpha}(\sigma), \Rr{\alpha}(\beta))$
to the collection of maps $(\Ss{\beta}(\sigma),
\Aa{\beta}(\beta))$. 
\begin{theorem}
\label{thm:refinement}  
Let $(\sigma,\alpha)$ be a collection of hypermaps with short hyperedges
and let $\beta$ be a refinement of $\alpha$. Then the collection of maps
 $(\Ss{\beta}(\sigma)), \Aa{\beta}(\beta))$ may be obtained
from $(\Ss{\alpha}(\sigma)), \Rr{\alpha}(\beta))$   by
performing the following transformations. 
\begin{enumerate}
\item We remove the singularities of $\beta$: the cycles $(i',k',j')\in
  S(\beta)$ and the points they contain.
\item For each cycle $(i,j,k)$ in $\alpha$ that is replaced by $(i,j)(k)$
  in $\beta$ we apply the contraction $(i,i')$, discard the buds $i'$
  and $i$ and replace the white point $j'$ with the black point $i$.  
\end{enumerate}  
\end{theorem}  
\begin{proof}
An important difference between $(\Ss{\alpha}(\sigma),
\Aa{\alpha}(\alpha))$ and $(\Ss{\beta}(\sigma),
\Aa{\beta}(\beta))$  is that $\Ss{\beta}(\sigma)$ is only a
subset of $\Ss{\alpha}(\sigma)$: the set of black vertices is the
same, but for each $3$-cycle $(i,j,k)$ of $\alpha$ that is replaced by a proper
refinement of $\beta$, the white vertex $(i',k',j')$ is only present in
$\Aa{\alpha}(\alpha)$ and must be removed from
$\Ss{\beta}(\sigma)$. The set of these white vertices is the set
$\Aa{\alpha}(\alpha)-\Ss{\beta}(\sigma)$.

Removing all singularities is a straightforward move: these form
isolated vertices in $(\Ss{\alpha}(\sigma),
\Rr{\alpha}(\beta))$. Consider now a $3$-cycle $(i,j,k)$ of
$\alpha$ that is replaced by $(i,j)(k)$ in $\beta$. An example of such a
cycle is the cycle $(7,8,6)$ in
Figure~\ref{fig:refinement}. Corresponding to such a refined cycle we
have the white vertex $(i',k',j')$, the edges $(i,i')$ and $(j,j')$ and
the bud $k'$ in $(\Ss{\alpha}(\sigma),
\Rr{\alpha}(\beta))$. Performing the contraction $(i,i')$ merges the white vertex $(i',k',j')$
with the black vertex containing $i$ as follows:
$$
(\ldots,\sigma^{-1}i,i',k',j',i,\ldots ) 
$$
After the contraction $(i,i')$ the points $i$ and $i'$ become buds, and
we still have the edge $(j,j')$. Discarding the buds $i,i'$ and $k'$ and
relabeling $j'$ as $i$ restores the cycle of $\sigma$ containing $i$ and
creates an edge $(i,j)$.
\end{proof}  

Combining Propositions~\ref{prop:samekappa} and \ref{prop:samefaces}
with Theorem~\ref{thm:refinement} we obtain the following formulas.

\begin{corollary}
Let $(\sigma,\alpha)$ be a collection of hypermaps with short hyperedges
and let $\beta$ be a refinement of $\alpha$. Then we have
\begin{align}
  \kappa(\Ss{\alpha}(\sigma),\Rr{\alpha}(\beta))&=
  \kappa(\sigma,\alpha)+|S(\beta)|\quad\mbox{and}\\ 
  z(\Rr{\alpha}(\beta)^{-1}\Ss{\alpha}(\sigma))&=
  z(\beta^{-1}\sigma)+|S(\beta)|. 
\end{align}  
\end{corollary}

\section{A formula for the number of spanning hypertrees}
\label{sec:sht}

In this section we prove a formula expressing the number of spanning
hypertrees of a hypermap $(\sigma,\alpha)$ with short hyperedges as a
weighted sum of the numbers of spanning trees of certain subgraphs of
the underlying graph of its associated two-colored map.  
To state our main result we introduce the following notation and
terminology.
\begin{definition}
Let $(\sigma,\alpha)$ be a hypermap with short hyperedges. We denote the
underlying graph of its associated two-colored map
$(\Ss{\alpha}(\sigma),\Aa{\alpha}(\alpha))$ with
$G(\sigma,\alpha)$ and call it the {\em associated two-colored graph of
  $(\sigma,\alpha)$}.  We denote the set of all, respectively white
vertices of $G(\sigma,\alpha)$  with $V(\sigma,\alpha)$, respectively
$W(\sigma,\alpha)$.  
\end{definition}  
\begin{definition}
Let $G$ be a graph on the vertex set $V$ and let $X$ be a subset of
$V$. We call the {\em restriction of $G$ to $X$} the graph whose vertex
set is $X$ and whose edges are exactly those edges of $G$ which are
incident to a pair of vertices contained in $X$. We denote the
restriction of $G$ by $\rest{G}{X}$.
\end{definition}
Using the above notation and terminology, the main result of this
section is the following.
\begin{theorem}
\label{thm:shtrees}
Let $(\sigma,\alpha)$ be a hypermap with short hyperedges and
$G(\sigma,\alpha)$ its associated two-colored
graph. Then the number of the spanning hypertrees of $(\sigma,\alpha)$
is given by 
$$
\sum_{S\subseteq W(\sigma,\alpha)} (-2)^{|S|}
s\left(\rest{G(\sigma,\alpha)}{V(\sigma,\alpha)-S}\right). 
$$
Here the function $s(G)$ associates to each graph $G$ the number of its spanning
trees.  
\end{theorem}
\begin{proof}
Consider a spanning hypertree $(\sigma,\theta)$ of $(\sigma,\alpha)$.
We define the subgraph $H(\sigma,\theta)$ of $G(\sigma,\alpha)$ as the
underlying graph of the collection of maps
$(\Ss{\alpha}(\sigma),\Rr{\alpha}(\theta))$. Equivalently,
$H(\sigma,\theta)$ is the subgraph obtained after deleting the following edges:
\begin{enumerate}
\item For each $2$-cycle $(i,j)$ of $\alpha$ which is replaced by the pair
  of fixpoints $(i)(j)$ in $\theta$ we delete the corresponding edge of
  $G(\Ss{\alpha}(\sigma),\Aa{\alpha}(\alpha))$.
\item For each $3$-cycle $(i,j,k)$ of $\alpha$ which is replaced by 
  $(i,j)(k)$ in $\theta$ we delete the edge $(k,k')$.
\item For each $3$-cycle $(i,j,k)$ of $\alpha$ which is replaced by the
  triplet of fixpoints $(i)(j)(k)$ in $\theta$ we delete the edges $(i,i')$,
  $(j,j')$ and $(k,k')$ in $G(\Ss{\alpha}(\sigma),\Aa{\alpha}(\alpha))$.  
\end{enumerate}  
An example of a spanning hypertree $(\sigma,\theta)$ of the hypermap
$(\sigma,\alpha)$ introduced in Example~\ref{ex:hypermap-g} and the
associated subgraph $H(\sigma,\theta)$ is shown in
Figure~\ref{fig:shtree}. Note that the white vertex $(1',3',2')$ has
become an isolated point, it belongs to the set $S(\theta)$ of
singularities of $\theta$. As a consequence
of the results in the preceding section, $(\sigma,\theta)$ is a spanning
hypertree of $(\sigma,\alpha)$ if and only if $H(\sigma,\theta)$ is a
spanning tree of the graph 
$\rest{G(\sigma,\alpha)}{V(\sigma,\alpha)-S(\theta)}$.
\begin{figure}[h]
%100%  
\begin{center}
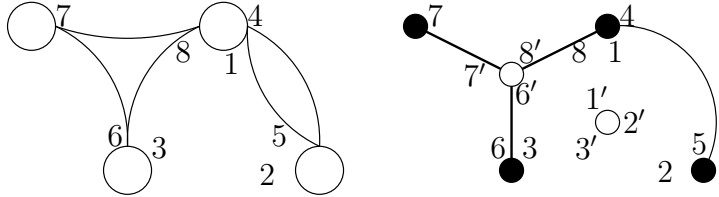
\end{center}
\caption{The collection of maps $(\Ss{\alpha}(\sigma),\theta'')$ for $\theta=(4,5)(6,7,8)$}
\label{fig:shtree} 
\end{figure}
For a fixed subset $S_1$ of the white vertices, consider the set of all
spanning hypertrees $(\sigma,\theta)$ such that $S(\theta)=S_1$. The map
$(\sigma,\theta)\mapsto 
\rest{H(\sigma,\theta)}{V(G(\sigma,\theta))-S_1}$ is injective, but
reaches only a subset of all spanning trees of
$\rest{G(\sigma,\theta)}{V(G(\sigma,\theta))-S_1}$: consider, once again
the hypermap $(\sigma,\alpha)$ introduced in
Example~\ref{ex:hypermap-g}, let us select $S_1=\emptyset$ and consider
the spanning tree of $G(\sigma,\theta)$ shown in
Figure~\ref{fig:unreachable}. The white vertex $(1',3',2')$  has become
a leaf, which is not possible according to our rules, which call for the
removal of zero, one, or three edges incident to a white vertex.    

\begin{figure}[h]
%100%  
\begin{center}
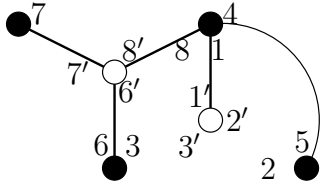
\end{center}
\caption{An unreachable spanning tree}
\label{fig:unreachable} 
\end{figure}
Given an arbitrary spanning tree of
$\rest{G(\sigma,\theta)}{V(G(\sigma,\theta))-S_1}$, let us denote the
set of its white leaves by $S_2$. The spanning tree corresponds to a
spanning hypertree of $(\sigma,\theta)$ if and only if
$S_2=\emptyset$. We count the number of such spanning trees by
inclusion-exclusion. Keeping in mind that the removal of all white
leaves from a spanning tree results in a spanning tree of the
restriction of $G(\sigma,\theta)$ to the remaining set of vertices, and that the
reattachment of each white leaf is possible in exactly three independent
ways, we obtain that
the number of all spanning hypertrees of $(\sigma,\theta)$ is given by
$$
\sum_{S_1\subseteq W(\sigma,\alpha)}\sum_{S_2\subseteq
  W(\sigma,\alpha)-S}  (-3)^{|S_2|}
s\left(\rest{G(\sigma,\alpha)}{V(\sigma,\alpha)-(S_1\cup S_2)}\right). 
$$
The statement now follows after introducing $S$ as the summation index
and using the formula $\sum_{S_2\subseteq S}
(-3)^{|S_2|}=(1-3)^{|S|}=(-2)^{|S|}$. 
\end{proof}

\section{Reciprocals of maps with maximum degree $3$}
\label{sec:3rec}

The formula provided by Theorem~\ref{thm:shtrees} provides a formula for
the number of spanning hypertrees in the reciprocal of a map
$(\sigma,\alpha)$ whose vertices have maximum degree $3$. This formula
only depends on the underlying graph of the map. By definition, the
associated two-colored graph $G(\alpha,\sigma)$ of the reciprocal
hypermap may be constructed as follows:
\begin{enumerate}
\item We associate a black vertex to each edge of $(\sigma,\alpha)$.
\item If two edges of $(\sigma,\alpha)$ are incident at a vertex of degree
  $2$, we connect them with an edge.
\item We associate a white vertex to each vertex of $(\sigma,\alpha)$
  that has degree $3$ and we connect them to the black vertices
  representing the incident edges.  
\end{enumerate}

\begin{figure}[h]
%100%  
\begin{center}
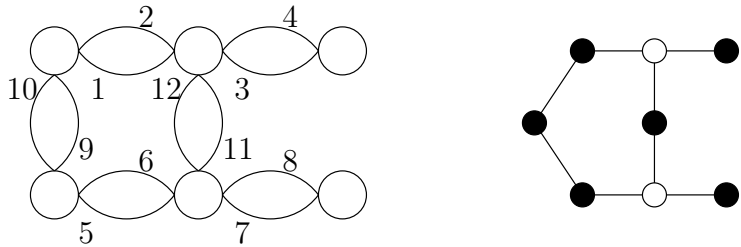
\end{center}
\caption{The map $(\sigma,\alpha)$ and the
  two-colored graph $G(\alpha,\sigma)$}
\label{fig:3rec} 
\end{figure}
Figure~\ref{fig:3rec} represents the map $(\sigma,\alpha)$ and the
two-colored graph $G(\alpha,\sigma)$ for
$\sigma=(1,10)(2,12,3)(4)(5,9)(6,7,11)(8)$ and
$\alpha=(1,2)(3,4)(5,6)(7,8),(9,10)(11,12)$. We selected the numbering of
points on the left as a helpful guide, but the process described above
would be the same for any map with the same underlying
graph. Theorem~\ref{thm:shtrees} has the following consequence 
\begin{corollary}
If $(\sigma,\alpha)$ is a map in which the maximum degree of a vertex is
$3$ then the number of spanning hypertrees in the reciprocal hypermap
$(\alpha,\sigma)$ depends only on the underlying graph of $(\sigma,\alpha)$.
\end{corollary}
In the application of Theorem~\ref{thm:shtrees} we do not need to
consider all subsets of the white vertices, because of the following
observation. 

\begin{lemma}
\label{lem:indep}
  Let $(\sigma,\alpha)$ be a map whose vertices have degree at most $3$
  and let $G(\alpha,\sigma)$ be the two-colored map associated to the
  reciprocal hypermap. If a subset $S$ of the white vertices satisfies
  that  $\rest{G(\sigma,\alpha)}{V(\sigma,\alpha)-S}$ is a connected
  graph then $S$ must correspond to an independent set of vertices in
  the underlying graph of $(\sigma,\alpha)$.
\end{lemma}
Indeed, if two white vertices representing adjacent vertices of degree $3$ of
$(\sigma,\alpha)$ both belong to $S$ then all black vertices connecting
them become isolated vertices after their removal, and the disconnected graph
$\rest{G(\sigma,\alpha)}{V(\sigma,\alpha)-S}$ has no spanning tree, only
spanning forests. The converse of Lemma~\ref{lem:indep} is not true in
general: the removal of any white vertex from the associated two-colored
graph shown in Figure~\ref{fig:3rec} results in a disconnected
graph. The construction of $G(\alpha,\sigma)$ has a very simple
description in the case when each vertex of $\sigma$ has degree $3$.
\begin{corollary}
Let $(\sigma,\alpha)$ be a map whose underlying graph is a $3$-regular
graph. Then the associated two-colored graph $G(\alpha,\sigma)$ of the
reciprocal hypermap may be obtained as follows:
\begin{enumerate}
\item We paint each vertex white.
\item We subdivide each edge into two edges by adding a black vertex.   
\end{enumerate}  
\end{corollary}
  
\begin{figure}[h]
%100%  
\begin{center}
\includegraphics{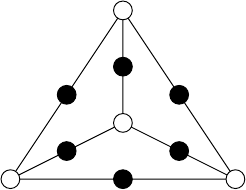}
\end{center}
\caption{$G(\alpha,\sigma)$ for $(\sigma,\alpha)$ whose underlying graph
is $K_4$}
\label{fig:k2rec} 
\end{figure}

\begin{example}
If the underlying graph of $(\sigma,\alpha)$ is the complete graph $K_4$
then the associated two-colored graph $G(\alpha,\sigma)$ is shown in
Figure~\ref{fig:k2rec}. This is a planar graph, we may draw it in the
plane even if $(\sigma,\alpha)$ does not have genus zero. Disregarding
the coloring of the vertices, the dual of the graph shown in
Figure~\ref{fig:k2rec} may be obtained by replacing each edge in $K_4$
with a pair of parallel edges. The number of spanning trees of this dual
graph is $2^3$ times the number of spanning trees of $K_4$, hence by
Cayley's theorem we obtain that the number of spanning trees in of the
dual of $G(\alpha,\sigma)$ is $2^3\times 4^{4-2}=128$. This is also the
number of spanning trees of $G(\alpha,\sigma)$. We may remove at most
one vertex of $G(\alpha,\sigma)$ without disconnecting it, and the
remaining graph has a unique circuit of length $6$. There are $4$ ways
to select the white vertex to be removed and the resulting graph has $6$
spanning trees. The total weight of these spanning trees is $(-2)\times
4\times 6=-48$. We obtain that the number of spanning hypertrees of
$(\alpha,\sigma)$ is $128-48=80$.    
\end{example}  

We devote the rest of this section to an application of
Theorem~\ref{thm:shtrees} to compute the number of spanning hypertrees
of the reciprocal of the map $(\sigma_m,\alpha_m)$, whose underlying
graph is the ``ladder'' graph with $m$ bounded faces. The map
$(\sigma_6,\alpha_6)$ and the associated two-colored graph
$G(\alpha_6,\sigma_6)$ is shown in Figure~\ref{fig:grid}.

\begin{figure}[h]
%100%  
\begin{center}
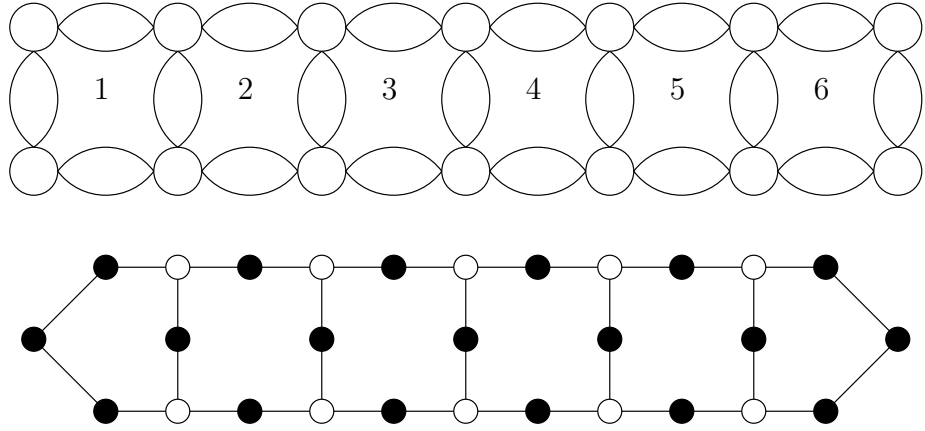
\end{center}
\caption{The map $(\sigma_6,\alpha_6)$ and the associated two-colored graph
$G(\alpha_6,\sigma_6)$}
\label{fig:grid} 
\end{figure}

We illustrate the effect of removing some white vertices in
Figure~\ref{fig:gridrec}. By Lemma~\ref{lem:indep} the set of white
vertices that we remove must correspond to a set of independent vertices in the
underlying graph of $(\sigma_m,\alpha_m)$. However, this condition is
not sufficient. The white vertices of $G(\alpha_m,\sigma_m)$ form $m-1$
columns which we number left to right, as shown
in Figure~\ref{fig:gridrec}. The removal of any pair of white vertices from two
consecutive columns disconnects $G(\alpha_m,\sigma_m)$, even if they
are selected from different rows, and hence are not adjacent.  
\begin{figure}[h]
%90%  
\begin{center}
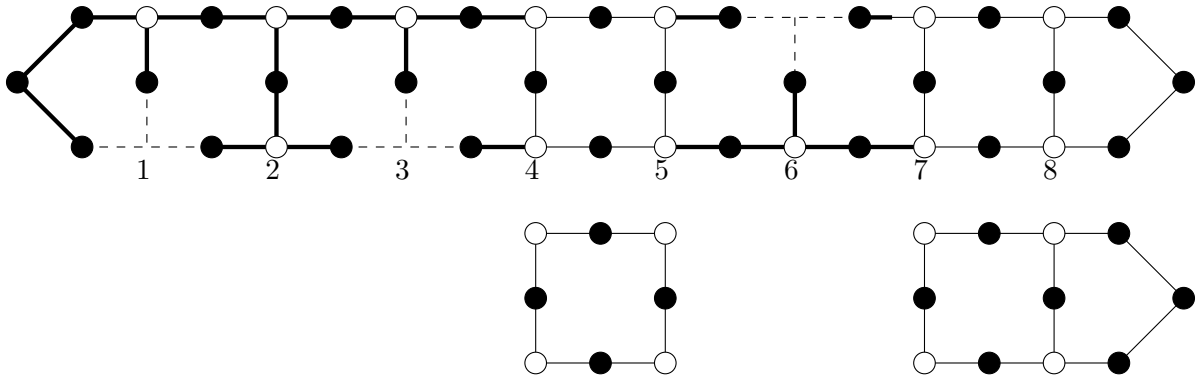
\end{center}
\caption{The result of removing some white vertices in $G(\alpha_9,\sigma_9)$}
\label{fig:gridrec}
\end{figure}
\begin{proposition}
\label{prop:lconnect}  
The removal of a set $S$ of white vertices does not disconnect
$G(\alpha_m,\sigma_m)$ if and only if $S$ satisfies the following two
conditions:
\begin{enumerate}
\item $S$ contains at most one element in each column.
\item There are no two elements of $S$ in neighboring columns.   
\end{enumerate}  
\end{proposition}  
The necessity of the conditions stated in
Proposition~\ref{prop:lconnect} has just been explained. The sufficiency
is also evident: if there is an intact column of white vertices next to
each removed white vertex, we can use the intact column to change levels
and go around the removed white vertices.
\begin{definition}
A set of positive integers is {\em sparse} if it contains no consecutive
integers. 
\end{definition}  
Using this terminology we can restate the second condition in
Proposition~\ref{prop:lconnect} as follows: the numbers of columns
selected to contain an element of $S$ must form a sparse subset of
$\{1,2,\ldots, m-1\}$.

In the example shown in
Figure~\ref{fig:gridrec} we removed a white vertex from the columns
$1,3,6$. The removal of these vertices forces the deletion of the edges
incident to them: these deleted edges are represented with dashed lines
in Figure~\ref{fig:gridrec}. The bold edges become bridges of the
remaining graph: they must be included in every spanning tree. The
remaining edges form disjoint subgraph: to select a spanning tree on the
whole graph amounts to independently selecting a spanning tree on each
of these subgraphs. In Figure~\ref{fig:gridrec} two such subgraphs
exist, and a separate copy of them is shown below the drawing
representing $\rest{G(\alpha_9,\sigma_9)}{V(\alpha_9,\sigma_9)-S}$. 

In general, after the removal of the deleted and bridge edges the
disjoint components of the remaining graph belong to one of the
following three types:

\begin{figure}[h]
%90%  
\begin{center}
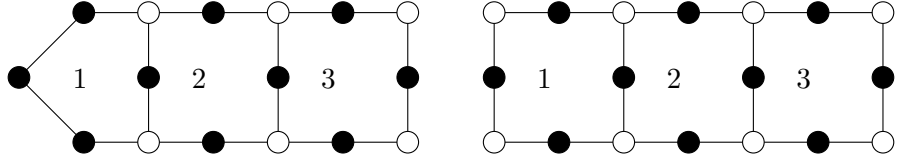
\end{center}
\caption{Components arising after the removal of the deleted and bridge edges}
\label{fig:types} 
\end{figure}

\begin{enumerate}
\item Graphs, whose bounded faces all have $8$ sides. An example of such
  a graph is shown on the left hand side of
  Figure~\ref{fig:types}. Let us denote with $F(k)$ the number of
  spanning trees of such graph having $k$ faces. (We set $F(0)=1$.)  
\item At the two extreme ends of $G(\alpha_m,\sigma_m)$ a copy of a
  graph similar to the one shown on the left hand side of
  Figure~\ref{fig:types} (or at the lower right end of
  Figure~\ref{fig:gridrec}) may arise. Among the bounded faces of such a graph
  have exactly one has $6$ sides and all the others have $8$ sides. Let
  us denote with $G(k)$ the number of spanning trees of such graph having
  $k$ faces. (We set $G(0)=1$.)  
\item Only when $S$ is the empty set, we are left to count all spanning
  trees of $G(\alpha_m,\sigma_m)$. Let us denote the number of its
  spanning trees with $H(m)$. 
\end{enumerate}  

Applying Theorem~\ref{thm:shtrees} to $G(\alpha_m,\sigma_m)$ we obtain the
following result.
\begin{proposition}
\label{prop:tmdec}
The number $T(m)$ of spanning hypertrees of $(\alpha_m,\sigma_m)$ is given by  
\begin{align*}
T(m) &=H(m)+\sum_{k=1}^{\lfloor (m-1)/2\rfloor}
(-4)^k\sum_{i_1<i_2<\cdots<i_k}
G(i_1-1)\cdot \prod_{j=1}^{k-1} F(i_{j+1}-i_j-2) \cdot G(m-i_k).
\end{align*}  
Here the summation $i_1<i_2<\cdots<i_k$ runs over all sparse subsets
$\{i_1,i_2,\ldots, i_k\}$ of $\{1,2,\ldots,m-1\}$. 
\end{proposition}  
\begin{proof}
The term $H(m)$ is contributed by the choice $S=\emptyset$. In all other
cases the numbers of the columns containing $S$ must form a sparse subset
of $\{1,2,\ldots,m-1\}$. After fixing the column numbers
$\{i_1,i_2,\ldots, i_k\}$ containing the $k\leq (m-1)/2$ elements of
$S$, we have two choices in each selected column to choose the upper or
the lower white vertex to be an element of $S$. This gives $2^k$
choices, which we multiply with the weight $(-2)^{|S|}=(-2)^k$ to get
the factor $(-4)^k$.

Deleting the edges incident to the white vertex $v$ in column number $i_j$
turns the remaining vertical edge above of below $v$ and all horizontal
edges between the columns numbered $i_{j}-1$ and $i_{j}+1$ into bridge
edges. Finally, if a white vertex is selected for removal in the first
or the last column then all remaining edges bordering the incident face
with $6$ sides turns into a bridge. The spanning trees in the
remaining graphs are counted by the product $G(i_1-1)\cdot
\prod_{j=1}^{k-1} F(i_{j+1}-i_j-2) \cdot G(m-i_k)$. 
\end{proof}
To compute the numbers $F(m)$, $G(m)$ and $H(m)$ observe that these
numbers all count spanning trees in planar graphs whose dual graph is
a ``generalized pencil'' as shown in Figure~\ref{fig:pencil}. 
\begin{figure}[h]
%90%  
\begin{center}
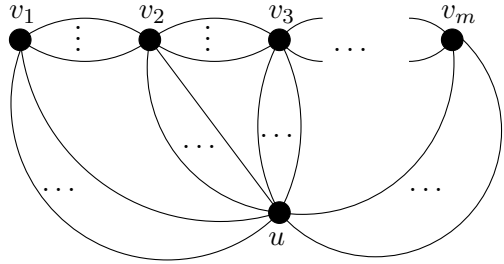
\end{center}
\caption{A generalized pencil graph}
\label{fig:pencil} 
\end{figure}
\begin{definition}
For $m\geq 1$ let $a_1,a_2,\ldots,a_m$ and $t$ be positive integers. We
define the {\em generalized pencil graph} 
$P(a_1,\ldots,a_m;t)$ as the following graph on the vertex 
set $\{u,v_1,v_2,\ldots,v_m\}$:
\begin{enumerate}
\item There are $a_i$ parallel edges between $u$ and $v_i$ for
  $i=1,\ldots,m$.
\item If $m\geq 2$ then there are $t$ parallel edges between $v_i$ and
  $v_{i+1}$ for $i=1,2,\ldots,m-1$.
\end{enumerate}  
\end{definition}
The following statements are direct consequences of the definition and
of the fact that a planar graph has the same number of spanning trees as
its dual.
\begin{corollary}
The following equalities hold:
\begin{enumerate}
\item  $F(m)$ is the number of spanning trees of
  $P(6,\underbrace{4,4,\ldots,4}_{m-2},6;2)$ for $m\geq 2$ and $F(1)=8$. 
\item $G(m)$ is the number of spanning trees of
$P(6,\underbrace{4,4,\ldots,4}_{m-1};2)$ for $m\geq 1$.
\item  $H(m)$ is the number of spanning trees of
  $P(\underbrace{4,4,\ldots,4}_{m};2)$ $m\geq 2$ and $H(1)=8$. . 
\end{enumerate}
  \end{corollary}  
\begin{proposition}
\label{prop:thetarec}  
The number $\Theta(a_1,\ldots,a_m;t)$ of the number of spanning trees of
$P(a_1,\ldots,a_m;t)$ satisfies the recurrence
$$
\Theta(a_1,\ldots,a_m;t)
=a_m\cdot \Theta(a_1,\ldots,a_{m-1};t)+ t\cdot
\Theta(a_1,\ldots,a_{m-2}, a_{m-1}+a_m;t)
$$
for $m\geq 2$.
\end{proposition}
\begin{proof}
If a spanning tree $T$ contains no edge between $v_{m-1}$ and $v_m$ then one of
the $a_m$ edges connecting $u$ and $v_{m}$ belongs to the spanning
tree. After the removal of the vertex $v_{m}$ the remaining edges of $T$
form a spanning tree of $P(a_1,\ldots,a_{m-1};t)$, which may be selected
independently.  

If a spanning tree $T$ contains one of the edges between $v_{m-1}$ and
$v_m$, then this edge may be selected in $t$ different ways. After
contracting this edge, the resulting tree is a spanning tree of a copy of
$P(a_1,\ldots,a_{m-2}, a_{m-1}+a_m;t)$, which may be selected independently. 
\end{proof}  
Using Proposition~\ref{prop:thetarec} it is not hard to prove the
following statement. 
\begin{proposition}
\label{prop:theta}  
The number $\Theta(a_1,\ldots,a_m;t)$ of the number of spanning trees of
$P(a_1,\ldots,a_m;t)$ is given by
\begin{equation}
\label{eq:Theta}  
\Theta(a_1,\ldots,a_m;t)
=\sum_{k=1}^{m} t^{m-k}\cdot \sum_{1\leq i_1<\cdots<i_k\leq m} \prod_{j=1}^{k-1}
(i_{j+1}-i_j)\cdot a_{i_1}\cdots a_{i_k}.
\end{equation}
\end{proposition}
\begin{proof}
Let us select a spanning tree $T$ by first selecting the set
$\{v_{i_1}, v_{i_2},\ldots, v_{i_k}\}$ of vertices connected
to $u$ by an edge. This set can not be empty, hence
we must have $k\geq 1$. The number of ways to select the edges incident
to $u$ is $a_{i_1}\cdots a_{i_k}$. Besides these $k$ edges there are
exactly $m-k$ edges in the spanning tree connecting $v_i$ and $v_{i+1}$ for some
$i\in\{1,2,\ldots,m-1\}$. To avoid $2$-cycles we may select at most one
edge between each $v_i$ and $v_{i+1}$ for our spanning tree.
Every vertex $v_i\not \in \{v_{i_1},
v_{i_2},\ldots, v_{i_k}\}$ must be reachable using edges of $T$ from a
unique element of $\{v_{i_1}, v_{i_2},\ldots, 
v_{i_k}\}$. If $i<i_1$ then this element can only be $v_{i_1}$, if
$i>i_k$ then this element can only be $v_{i_k}$. In the remaining cases
$i_j<i<i_{j+1}$ holds for some $j$ and $v_{i}$ is either reachable from
$v_{i_j}$ or from $v_{i_{j+1}}$. To avoid creating a cycle and keep $T$
connected,  there must be exactly one $i\in \{i_j, i_j+1,\ldots
i_{j+1}-1\}$ such that there is no edge between $v_i$ and $v_{i+1}$ in $T$.  
This $i$ may be selected in $(i_{j+1}-i_j)$ ways and this choice may be
performed independently for $j=1,2,\ldots,k-1$. After making these
choices, there are $t^{m-k}$ ways to select the remaining edges of $T$.
\end{proof}  

Next we use Proposition~\ref{prop:theta} to compute $F(m), G(m)$ and
$H(m)$. In our computation, the following Lemma will play a key role.
\begin{lemma}
  \label{lem:productbin}
The following identities hold for all $k,m\geq 2$:  
\begin{equation}
\label{eq:e1}
\sum_{1<i_1<i_2<\cdots<i_k<m} \prod_{j=1}^{k-1}
(i_{j+1}-i_j)=\binom{m+k-3}{2k-1}.
\end{equation}
\begin{equation}
\label{eq:e2}
\sum_{1=i_1<i_2<\cdots<i_k< m}
\prod_{j=1}^{k-1} (i_{j+1}-i_j)=\binom{m+k-3}{2k-2}
\end{equation}
\begin{equation}
\label{eq:e3}
\sum_{1<i_1<i_2<\cdots<i_k= m}
\prod_{j=1}^{k-1} (i_{j+1}-i_j)=\binom{m+k-3}{2k-2}
\end{equation}
\begin{equation}
\label{eq:e4}
\sum_{1=i_1<i_2<\cdots<i_k= m}
\prod_{j=1}^{k-1} (i_{j+1}-i_j)=\binom{m+k-3}{2k-3}
\end{equation}
\end{lemma}
\begin{proof}
We provide all details of \eqref{eq:e1}, the proof of the remaining
equalities is completely analogous.
We may select a $2k-1$ element subset $\{j_1,\ldots,j_{2k-1}\}$
satisfying $2\leq j_1<\cdots<j_{2k-1}\leq m+k-2$  of
$\{2,\ldots,m+k-2\}$ in two steps, as follows:
\begin{enumerate}
\item First we select the odd indexed entries $j_1, j_3,\ldots j_{2k-1}$.
\item  Introducing $i_1=j_1$, $i_2=j_3-1$, \ldots
  $i_k=j_{2k-1}-k+1$, there are $i_2-i_1$ ways to select $j_2$, $i_3-i_3$ ways
  to select $j_4$, and so on, $i_k-i_{k-1}$ ways to select $j_{2k-2}$.  
\end{enumerate}  
To prove the remaining identities we can follow the same procedure,
after making some adjustments. In~\eqref{eq:e2} we are
selecting a $2k-1$ element subset of $\{1,\ldots,m+k-2\}$ 
but we are fixing $j_1=1$, hence we are freely selecting $2k-2$ elements
from the same set $\{2,\ldots,m+k-2\}$ . Similarly, in~\eqref{eq:e3} we
are fixing $j_{2k-1}=m+k-1$ and we are selecting $2k-2$ elements out of
$m+k-3$ elements. Finally in in~\eqref{eq:e4} we are selecting both
$j_1=1$ and $j_{2k-1}=m+k-1$ outside the set $\{2,\ldots,m+k-2\}$ and
then we are selecting freely the remaining $2k-3$ elements.
\end{proof}  

\begin{proposition}
\label{prop:Thetamain}
For a fixed pair of positive integers $(a,b)$ let $\Phi(m)$ be the function 
$$
\Phi(m)=
\begin{cases}
  1 & \text{if $m=0$,}\\
  (a+b)x & \text{if $m=1$,}\\
  \Theta(a,\underbrace{4,\ldots,4}_{m-2},b;2) & \text{if $m\geq 2$.}\\
\end{cases}  
$$
Then 
\begin{align*}
  \sum_{m\geq 0} \Phi(m)\cdot x^m&= 1+\frac{x^2(16-2a-2b)+(a+b-4)x}{(1-2x)^2}\\
  &+
\left(64x^2+8(a+b)x(1-2x)
+ab(1-2x)^2 \right)\cdot \frac{x^2}{(4x^2 - 8x + 1)(1 - 2x)^2}  
\end{align*}
\end{proposition}  
\begin{proof}
 We apply Proposition~\ref{prop:theta} to compute $\Phi(m)$ for
 $m\geq$ 2. For this purpose we substitute $t=2$.  The sum of all terms
 satisfying $k=1$ on the right hand side of \eqref{eq:Theta} is
 $2^{m-1}(a+(m-2)4+b)$, and their contribution to $\sum_{m\geq 0}
 \Phi(m)x^m$ is
 \begin{align}
   \nonumber
   \sum_{m\geq 1} (a+b-8+4m)\cdot 2^{m-1}x^m&=
   \frac{(a+b-8)x}{1-2x}+\frac{4x}{(1-2x)^2}\\
   &=\frac{x^2(16-2a-2b)+(a+b-4)x}{(1-2x)^2}.
 \end{align}

 Consider now the contribution of all terms
 satisfying $k\geq 2$. To decide what to
 substitute into the factors $a_{i_j}$ we need to distinguish four cases.

 {\bf\noindent Case 1:} $1<i_1<i_k<m$ holds. In this case $a_{i_1}\cdots
a_{i_k}=4^k$. By~\eqref{eq:e1}, all terms covered in
this case contribute
$$
2^{m-k}4^k\cdot \sum_{1< i_1<\cdots<i_k< m} \prod_{j=1}^{k-1}
(i_{j+1}-i_j)=2^{m+k}\binom{m+k-3}{2k-1}. 
$$
For a fixed $k\geq 2$ the contribution of all terms to $\sum_{m\geq 0}
 \Phi(m)x^m$ covered in this case is
 \begin{align*}
\sum_{m\geq k+2}& 2^{m+k}\binom{m+k-3}{2k-1}\cdot x^m
=\frac{2^k}{(2x)^{k-3}} \sum_{m\geq k+2} \binom{m+k-3}{2k-1}\cdot
(2x)^{m+k-3}\\
&=\frac{8}{x^{k-3}}\cdot
\frac{(2x)^{2k-1}}{(1-2x)^{2k}}=\frac{2^{2k+2}x^{k+2}}{(1-2x)^{2k}}=4x^2\cdot
\left(\frac{4x}{(1-2x)^2}\right)^k.
\end{align*}

{\bf\noindent Case 2:} $1=i_1<i_k<m$ holds. In this case $a_{i_1}\cdots
a_{i_k}=a\cdot 4^{k-1}$. By~\eqref{eq:e2}, all terms covered in
this case contribute
$$
2^{m-k}\cdot a\cdot 4^{k-1}\cdot \sum_{1=i_1<i_2<\cdots<i_k< m}
\prod_{j=1}^{k-1} (i_{j+1}-i_j)=a\cdot 2^{m+k-2}\binom{m+k-3}{2k-2}. 
$$
For a fixed $k\geq 2$ the contribution of all terms to $\sum_{m\geq 0}
 \Phi(m)x^m$ covered in this case is
 \begin{align*}
\sum_{m\geq k+1}& a\cdot 2^{m+k-2}\binom{m+k-3}{2k-2} \cdot x^m
=\frac{a\cdot  2^{k-2}}{(2x)^{k-3}} \sum_{m\geq k+1} \binom{m+k-3}{2k-2}\cdot
(2x)^{m+k-3}\\ 
&=\frac{2a}{x^{k-3}}\cdot \frac{(2x)^{2k-2}}{(1-2x)^{2k-1}}=\frac{a\cdot
  2^{2k-1}x^{k+1}}{(1-2x)^{2k-1}}=\frac{ax(1-2x)}{2}\cdot
\left(\frac{4x}{(1-2x)^2}\right)^k.
\end{align*}

{\bf\noindent Case 3:} $1<i_1<i_k=m$ holds. In this case $a_{i_1}\cdots
a_{i_k}=b\cdot 4^{k-1}$. By~\eqref{eq:e3}, all terms covered in
this case contribute
$$
2^{m-k}\cdot b\cdot 4^{k-1}\cdot \sum_{1<i_1<i_2<\cdots<i_k=m}
\prod_{j=1}^{k-1} (i_{j+1}-i_j)=b\cdot 2^{m+k-2}\binom{m+k-3}{2k-2}. 
$$
Similarly to the previous case, for a fixed $k\geq 2$ the contribution
of all terms to $\sum_{m\geq 0}\Phi(m)x^m$ covered in this case is
 \begin{align*}
\sum_{m\geq k+1} b\cdot 2^{m+k-2}\binom{m+k-3}{2k-2} \cdot x^m
&=\frac{bx(1-2x)}{2}\cdot
\left(\frac{4x}{(1-2x)^2}\right)^k.
\end{align*}

{\bf\noindent Case 4:} $1=i_1<i_k=m$ holds. In this case $a_{i_1}\cdots
a_{i_k}=ab\cdot 4^{k-2}$. By~\eqref{eq:e4}, all terms covered in
this case contribute
$$
2^{m-k}\cdot ab\cdot 4^{k-2}\cdot \sum_{1=i_1<i_2<\cdots<i_k=m}
\prod_{j=1}^{k-1} (i_{j+1}-i_j)=ab\cdot 2^{m+k-4}\binom{m+k-3}{2k-3}. 
$$
For a fixed $k\geq 2$ the contribution of all terms to $\sum_{m\geq 0}
 \Phi(m)x^m$ covered in this case is
 \begin{align*}
\sum_{m\geq k}& ab\cdot 2^{m+k-4}\binom{m+k-3}{2k-3}\cdot x^m
=\frac{ab\cdot 2^{k-4}}{(2x)^{k-3}}\cdot \sum_{m\geq k} \binom{m+k-3}{2k-3}\cdot
(2x)^{m+k-3}\\
&=\frac{ab\cdot 2^{k-4}}{(2x)^{k-3}}\cdot
\frac{(2x)^{2k-3}}{(1-2x)^{2k-2}}=\frac{ab\cdot
  2^{2k-4}x^{k}}{(1-2x)^{2k-2}}=
\frac{ab(1-2x)^2}{16}\cdot
\left(\frac{4x}{(1-2x)^2}\right)^k.
\end{align*}

Summing over the contributions of all $k\geq 2$ in all four cases yields
\begin{align}
\label{eq:kgeq2}  
 \left(4x^2+\frac{(a+b)x(1-2x)}{2}+\frac{ab(1-2x)^2}{16} \right)\cdot   \sum_{k\geq 2} \left(\frac{4x}{(1-2x)^2}\right)^k
  \end{align}  
Here
\begin{align*}
\sum_{k\geq 2} \left(\frac{4x}{(1-2x)^2}\right)^k&=
\frac{1}{1-\frac{4x}{(1-2x)^2}}-1-\frac{4x}{(1-2x)^2}\\
   &=\frac{16x^2}{(4x^2 - 8x + 1)(1 - 2x)^2}.
\end{align*}
\end{proof}  
Substituting $a=6$ and $b=6$ into Proposition~\ref{prop:Thetamain}
we obtain that the generating function for the numbers $F(m)$ is
\begin{equation}
\label{eq:fgfun}
\FF(x)=\sum_{m\geq 0} F(m)\cdot x^m=\frac{1}{1-8x+4x^2}.
\end{equation}  
Similarly, substituting $a=6$ and $b=4$ into Proposition~\ref{prop:Thetamain}
yields the generating function for the numbers $G(m)$:
\begin{equation}
\label{eq:ggfun}
\GG(x)=\sum_{m\geq 0} G(m)\cdot x^m=\frac{(1-2x)}{1-8x+4x^2}.
\end{equation}  
Finally, substituting $a=4$ and $b=4$ into
Proposition~\ref{prop:Thetamain} yields the generating function for the
numbers $H(m)$: 
\begin{equation}
\label{eq:hgfun}
\HH(x)=\sum_{m\geq 0} H(m)\cdot x^m=\frac{(1-2x)^2}{1-8x+4x^2}.
\end{equation}  
\begin{remark}
The numbers $F(m)$ are listed as sequence A099156 in the
OEIS~\cite{OEIS}.
They are also given by the formula $F(m)=2^m\cdot U_m(2)$ where $U_m(x)$ is
the $m$-th Tchebyshev polynomial of the second kind. The numbers $G(m)$
are listed as sequence A102591 in the OEIS~\cite{OEIS}. They are also
given by the formula:
$$
G(m)=\sum_{k=0}^m \binom{2m+1}{2k}\cdot 3^{m-k}.
$$
\end{remark}  
\begin{theorem}
  The generating function of the
  number $T(m)$ of spanning hypertrees of $(\alpha_m,\sigma_m)$ is given
  by
  $$
\TT(x)=\sum_{m\geq 0} T(m)\cdot x^m=\frac{(1 - 2x)^2}{1-8x+8x^2}.
  $$
\end{theorem}
\begin{proof}
  By Proposition~\ref{prop:tmdec} we have
  \begin{align}
  \label{tfun}  
    \TT(x)&=\HH(x)+\sum_{k\geq 1}
    (-4x^2)^k\cdot  \GG(x)\cdot \FF(x)^{k-1} \cdot \GG(x)
    =\HH(x) -\frac{4x^2\cdot \GG(x)^2}{1+4x^2\cdot \FF(x)}.
\end{align}    
Substituting the formulas \eqref{eq:fgfun}, \eqref{eq:ggfun} and
\eqref{eq:hgfun} into \eqref{tfun} yields the stated formula.
\end{proof}  
\begin{remark}
The sequence $(T(m)/2^m)_{m= 0}^{\infty}$ is sequence A003480 in the
OEIS~\cite{OEIS} and it is the solution of several combinatorial problems. 
\end{remark}  

\section{Concluding remarks}

Part of our work on the Whitney polynomial~\cite{Cori-Hetyei2} focuses
on constructing the directed medial map of a hypermap and generalizing
or proving analogues of several
results of Arratia, Bollob\'as, Ellis-Monaghan, Martin and
Sorkin~\cite{Arratia-Bollobas,Bollobas-cp,Ellis-Monaghan-mp,Ellis-Monaghan-mpmisc,Ellis-Monaghan-cpp,Ellis-Monaghan-exploring,Martin-Thesis,Martin}
on the circuit partition polynomials of Eulerian digraphs and the medial
graph of a plane graph. As pointed out
in~\cite[Remark~5.3]{Cori-Hetyei2}, a key difference between our
generalizations and the ones present in the above cited literature is, that
we have to consider hypermaps embedded in a surface and the {\em
  Eulerian states} we define must be {\em noncrossing}. This noncrossing
condition is automatically satisfied in the case of maps, but needs
verification for hyperedges of length greater than $2$. That said,
our directed medial map constructions calls for associating the cycle
$(1^-,1^+,2^-,2^+,3^-,3^+)$ to the hyperedge $(1,2,3)$ and then selecting a
noncrossing matching on the set $\{1^-,1^+,2^-,2^+,3^-,3^+\}$ that matches
each point with a point of opposite sign. Out of the $3!=6$ matchings
between points of opposite signs exactly one is crossing. It seems
reasonable to conjecture that the standard results on circuit partition
polynomials and our results may be more closely related in the case of
hypermaps with short hyperedges: perhaps the use of the associated
$2$-colored map introduced in Section~\ref{sec:a2cm} could help express
this connection explicitly.

We tried and failed to generalize Bernardi's
result~\cite{Bernardi-embeddings} on a topological definition of a
Tutte polynomial to our hypermap setting. A plausible generalization of the {\em
  Bernardi tour}  of a spanning tree in a topological graph to hypermaps
and spanning hypertrees is not hard
to find, and we have done so in~\cite{Cori-Hetyei} by taking the dual of
the tour described by Cori~\cite{Cori-hrec} and
  Mach\`\i~\cite{Machi}. If however we try to
generalize Bernardi's definition of activities to even hypermaps with
short hyperedges, we seem to be unable to avoid dependence on the root
selection and the labeling of the points.  

Corollary~\ref{cor:almostTutte} inspires the question of finding a
recursively defined class of hypermaps with short hyperedges whose
Whitney polynomial may be computed without encountering a double loop or
double bridge in the process. We may define the Tutte polynomial of
such hypermaps by $T(\sigma,\alpha;x,y)=R(\sigma,\alpha;x-1,y-1)$ and
attempt to find a combinatorial description of its nonnegative integer
coefficients.

\section*{Acknowledgments}
The second author wishes to express his heartfelt thanks to Labri, Universit\'e
Bordeaux I, for hosting him as a visiting researcher in Spring 2024,
where a great part of this research was performed. This 
work was partially supported by a grant from the Simons Foundation 
(\#514648 to G\'abor Hetyei).

\end{document}